\documentclass[dvipsnames]{amsart}
\usepackage[utf8]{inputenc}
\usepackage{graphicx}
\usepackage{amssymb}
\usepackage{amsmath,tensor}
\usepackage{float}
\usepackage{xcolor}
\usepackage[
style=alphabetic,
maxnames=5
]{biblatex}
\usepackage{mathscinet}

\usepackage{hyperref}
\usepackage{cleveref}
\usepackage{comment}
\usepackage[all]{xy}
\usepackage{tikz-cd}

\usepackage[dvips]{epsfig}
\usepackage{pinlabel}

\newcommand{\ig}[2]{\vcenter{\xy (0,0)*{\includegraphics[scale=#1]{./EPS/#2}} \endxy}}
\newcommand{\igs}[1]{\ig{.6}{#1}}
\newcommand{\igm}[1]{\ig{.9}{#1}}
\newcommand{\igb}[1]{\ig{1.1}{#1}}

\newtheorem{lettertheorem}{Theorem}

\newtheorem{thm}{Theorem}[section] 
\newtheorem{prop}[thm]{Proposition}
\newtheorem{lem}[thm]{Lemma}
\newtheorem{cor}[thm]{Corollary}

\theoremstyle{definition}
\newtheorem{defn}[thm]{Definition}

\newtheorem{ex}[thm]{Example}

\theoremstyle{remark}
\newtheorem{rem}[thm]{Remark}

\newcommand{\eb}{\mathbf{e}}

\def\Z{\mathbb{Z}}
\def\Q{\mathbb{Q}}

\def\hat{\widehat}
\newcommand{\ot}{\otimes}
\newcommand{\co}{\colon}
\newcommand{\pa}{\partial}

\DeclareMathOperator{\Hom}{Hom}
\DeclareMathOperator{\End}{End}

\DeclareMathOperator{\id}{id}

\DeclareMathOperator{\Id}{Id}
\newcommand{\im}{\operatorname{Im}}

\newcommand{\at}{{\mathtt{a}}}

\newcommand{\mi}{\underline}
\newcommand{\ma}{\overline}
\newcommand{\expr}{\leftrightharpoons}

\newcommand{\mt}{\emptyset}

\newcommand{\Frob}{\mathbf{Frob}}

\newcommand{\SBSBim}{\mathbf{SBSBim}}

\newcommand{\simto}{\stackrel{\sim}{\to}}
\newcommand{\LL}{\mathbb{LL}}
\newcommand{\calD}{\mathcal{D}}

\newcommand{\ELL}{\mathbb{ELL}}

\newcommand{\DLL}{\mathbb{DLL}}

\DeclareMathOperator{\BS}{BS}

\DeclareMathOperator{\leftred}{LR}
\DeclareMathOperator{\rightred}{RR}
\DeclareMathOperator{\core}{core}

\DeclareMathOperator{\quot}{qu}
\DeclareMathOperator{\PF}{DL}
\DeclareMathOperator{\Span}{Span}
\DeclareMathOperator{\new}{new}

\newcommand{\symm}{\mathbf{S}}

\newcommand{\ffk}{{\mathbb{F}}}

\newcommand{\sberry}{\textcolor{WildStrawberry}{s}}
\newcommand{\teal}{\textcolor{teal}{t}}

\title{The atomic Leibniz rule}
\author[Ben Elias]{Ben Elias}
\address{Department of Mathematics, Fenton Hall, University of Oregon,
	Eugene, OR, 97403-1222, USA}
\email{belias@uoregon.edu}

\author[Hankyung Ko]{Hankyung Ko}
\address{Department of Mathematics, Uppsala University,
	Box. 480,
	SE-75106, Uppsala, Sweden}
\email{hankyung.ko@math.uu.se}

\author[Nicolas Libedinsky]{Nicolas Libedinsky}
\address{Departamento de Matemáticas, Facultad de Ciencias, Universidad de Chile}
\email{nlibedinsky@gmail.com}

\author[Leonardo Patimo]{Leonardo Patimo}
\address{Dipartimento di Matematica, Universit\`a di Pisa, Italy}
\email{leonardo.patimo@unipi.it}

\begin{document}
	
\begin{abstract} The Demazure operator associated to a simple reflection satisfies the twisted Leibniz rule. In this paper we introduce a generalization of the twisted Leibniz rule for the Demazure operator associated to any atomic double coset. We prove that this atomic Leibniz rule is equivalent to a polynomial forcing property for singular Soergel bimodules. \end{abstract}

\maketitle

\section{Introduction}

We introduce a generalization of the twisted Leibniz rule which applies to Demazure operators associated with certain double cosets. We call it the atomic Leibniz rule,  and it should play an important role in an eventual description of the singular Hecke category by generators and relations. This result is the algebraic heart which we extract from singular Soergel bimodules and transplant to their diagrammatic calculus.

\subsection{}
 In this paper, we work with general Coxeter systems and quite general actions of these groups on polynomial rings. For ease of exposition,  in this section, we assume that  $R = \Q[x_1, \ldots, x_n]$, with the standard action of the symmetric group $W = \symm_n$. Let $s_i$ be the simple reflection swapping $i$ and $i+1$, and let $S = \{s_1, \ldots, s_{n-1}\}$ be the set of simple reflections. There is a \emph{Demazure operator} $\pa_{s_i} \co R \to R$ defined by the formula
\begin{equation} \label{demazuredefn} \pa_{s_i}(f) = \frac{f - s_i f}{x_i - x_{i+1}}. \end{equation}
This operator famously satisfies a \emph{twisted Leibniz rule}
\begin{equation}\label{eq:twistedLeibniz} \pa_{s_i}(fg) = s_i(f) \pa_{s_i}(g) + \pa_{s_i}(f) g. \end{equation}
	
Because the operators $\pa_{s_i}$ satisfy the braid relations, one can unambiguously define an operator $\pa_w$ associated with any $w \in \symm_n$ by composing the operators $\pa_{s_i}$ along a reduced expression. When computing $\pa_w(fg)$, one could apply the twisted Leibniz rule repeatedly to obtain a complicated generalization of \eqref{eq:twistedLeibniz} for $\pa_w$. We discuss this in \S\ref{ssec:leibnizpermutations}.
	
What we provide in this paper is the natural generalization of the twisted Leibniz rule to the setting of double cosets. In double coset combinatorics, the analogues of simple reflections are called atomic cosets. We prove a version of \eqref{eq:twistedLeibniz} for Demazure operators attached to atomic cosets. 

	
\subsection{} We recall some definitions so as to precisely state our first theorem. Let $(W,S)$ be any Coxeter system. For any subset $I \subset S$, let $W_I$ be the parabolic subgroup of $W$ generated by $I$. We assume that $I$ is \emph{finitary}, i.e. that $W_I$ is a finite group.    Let $R^I \subset R$ be the subring of polynomials invariant under $W_I$. Let $w_I$ denote the longest element of $W_I$. For two subsets $I, J \subset S$, a double coset $p \in W_I \backslash W / W_J$ will be called an \emph{$(I,J)$-coset}. Any $(I,J)$-coset $p$ has a unique minimal element $\mi{p} \in W$ and a unique maximal element $\ma{p} \in W$ with respect to the Bruhat order.
	
In \cite[\S 3.4]{KELP1}, we introduced a Demazure operator $\pa_p : R^J \to R^I$ for any $(I,J)$-coset $p$. In fact, $\pa_p$ is equal to $\pa_w$ for $w = \ma{p} w_J \in W$. Normally one views $\pa_w$ as a map $R \to R$, 
but when the source of this map is restricted from $R$ to $R^J$, then the image is contained in $R^I$. When $I = J = \mt$, so that $p = \{w\}$ for some $w \in W$, then $\pa_p = \pa_w$.
	
In \cite{KELP3} we introduced the notion of an \emph{atomic} $(I,J)$-coset.
Briefly stated, $\at$ is atomic if \begin{itemize}
\item there exists $s, t \in S$ (possibly equal) with $s \notin I$ and $t \notin J$ such that $I \cup \{s\} = J \cup \{t\} =: M$.
\item $W_M$ is a finite group, $\ma{\at} = w_M$ and $w_M s = t w_M$.
\end{itemize}
In particular, the subsets $I$ and $J$ are conjugate under both $\ma{\at}$ and $\mi{\at}$. If $f \in R^J$ then $\mi{\at}(f) \in R^I$. We also have $\mi{\at} = \ma{\at} w_J$, so that $\pa_\at$ is the restriction of  $\pa_{\mi{\at}}$. 
	
When $I = J = \mt$, the atomic $(I,J)$-cosets have the form $\{s\}$ for $s \in S$. Atomic cosets in $\symm_n$ can be described using cabled crossings. For example, the permutation with one-line notation $(345671289)$ crosses the first two numbers past the next five while fixing the remaining numbers. It is a prototypical example of $\mi{\at}$ for an atomic coset $\at$; in this example $W_I \cong \symm_2 \times \symm_5 \times \cdots$ and $W_J \cong \symm_5 \times \symm_2 \times \cdots$ and $W_M \cong \symm_7 \times \cdots$. 

\subsection{}

Our story began with the defining representation of $\symm_n$ over $\Q$. Then we considered the polynomial ring $R$, its invariant subrings, and the structure derived therefrom. Alternatively, one can start with a \emph{realization} of a Coxeter system $(W,S)$, which is effectively a ``reflection representation'' of $W$ over a commutative base ring $\Bbbk$, equipped with a choice of roots and coroots. From this representation we construct its polynomial ring $R$, Demazure operators, etcetera. See \S\ref{ss.realizations} for details.

For each realization\footnote{Soergel's construction depends only on a reflection representation of $W$, and not a choice of roots and coroots. Elias and Williamson gave a presentation of Soergel's category in \cite{Soergelcalculus} which depends on a choice of realization. The Demazure operators also depend on the choice of roots and coroots.}, Soergel defined a monoidal category of $R$-bimodules called Soergel bimodules, and proved that for certain realizations\footnote{Specifically, the base ring $\Bbbk$ should be an infinite field of characteristic not equal to $2$, and the representation should be (faithful and) reflection-faithful.} (that we call \emph{SW-realizations})
one obtains a categorification of the Hecke algebra. Soergel bimodules and related categorifications of the Hecke algebra are objects of critical importance in geometric representation theory. The following is the first main result of this paper (Theorem~\ref{thm:main}).
\begin{lettertheorem}[The atomic Leibniz rule] \label{thm:alr} 
 Let $(W,S)$ be any Coxeter system equipped with an  SW-realization, or the symmetric group  $\symm_n$ with $R=\mathbb{Z}[x_1,\ldots,x_n].$ 
 Let $\at$ be an atomic $(I,J)$-coset and $\leq$ the Bruhat order for double cosets\footnote{The Bruhat order on $(I,J)$-cosets can be defined by $q \le \at$ if and only if $\mi{q} \le \mi{\at}$. See \cite[Thm. 2.16]{KELP2} for equivalent definitions. 
 Note that $q$ need not be atomic, so $\pa_q$ need not equal (the restriction of) $\pa_{\mi{q}}$.}. For every  $q<\at$ 
 there is a unique $R^{I\cup J}$-linear operator on polynomials denoted  $T^\at_q$, satisfying the equation
\begin{equation}\label{atomicprototype}
\pa_\at(f \cdot g) = \mi{\at}(f) \pa_\at(g) + \sum_{q < \at} \pa_q(T^\at_q(f) \cdot g) \end{equation}
for all  $f,g\in R^J.$ Polynomials in the image of $T^\at_q$ are appropriately invariant, see Definition \ref{defn:atomicLeibniz} for details. When $\at$ is fixed, we will often write $T_q := T_q^\at$.
\end{lettertheorem}

 The motivation for this result will take some setup, so please bear with us. 
	
\begin{ex}\label{ex:p=s}  We let $e$ denote the identity element. Let $s$ be a simple reflection and let $\at$ be the $(\mt,\mt)$-coset $\{s\}$. There is only one coset less than $\at$, namely $q = \{e\}$.  We set $T^\at_q(f) = \pa_s(f)$. Note that $\pa_\at = \pa_s$ and $\pa_q = \id$. Then \eqref{atomicprototype} becomes
\begin{equation} \pa_s(fg) = s(f) \pa_s(g) + \pa_s(f) g, \end{equation}
which recovers the twisted Leibniz rule. \end{ex}
  
For more examples see \S\ref{subsec:examplesintro}. There we also give an example of a non-atomic coset $p$ for which no equality of the form \eqref{atomicprototype} can hold (with $p$ replacing $\at$).
	
\begin{rem} Though we can prove the existence and uniqueness of such a formula, we do not have an explicit description of the operators $T_q$. We consider this an interesting open problem.  A more accessible problem is to compute $T_q$ on a (carefully chosen) set of generators of $R^J$, which we accomplish in type $A$ in Theorem \ref{thm:demazureoncompletedelta}. This is useful computationally because it gives enough information
to apply the atomic Leibniz rule for any pair of elements in $R^J$ (see the discussion of multiplicativity in \S\ref{ssec:lastpartofintro}).  \end{rem}

\begin{rem} \label{rem:alrawesome}  In \S\ref{ssec:changerealization} we prove that the atomic Leibniz rule for one realization implies the atomic Leibniz rule for related realizations (e.g. obtained via base change, enlargement, or quotient). The atomic Leibniz rule also depends only on the restriction of the realization to the finite parabolic subgroup $W_M$ associated to $\at$. Thus Theorem \ref{thm:alr} implies the atomic Leibniz rule in broad generality for realizations of Coxeter systems in both finite and affine type $A$, and in finite characteristic as well, see Example \ref{ex:atomicforaffine}.
\end{rem}

\subsection{}

Our second main result is a connection between the atomic Leibniz rule and the theory of singular Soergel bimodules  \cite{SingSb}.  Like Soergel bimodules, singular Soergel bimodules are ubiquitous in geometric representation theory, appearing e.g. in the geometric Satake equivalence and other situations where partial flag varieties play a role.
Specifically, we connect the atomic Leibniz rule to a property called polynomial forcing, whose motivation we postpone a little longer.

Let us first be more precise. Singular Soergel bimodules are graded bimodules over graded rings, but we ignore all gradings in this paper.
To an atomic coset $\at$ as above, one can associate the $(R^I, R^J)$-bimodule  $B_\at := R^I \ot_{R^M} R^J$. This is an indecomposable bimodule; there are also indecomposable $(R^I, R^J)$-bimodules $B_q$ associated to any $(I,J)$-coset $q$, which we will not try to describe here. Then $B_\at$ has a submodule called the \emph{submodule of lower terms}, spanned by the images of all bimodule endomorphisms of $B_\at$ which factor through $B_q$ for $q < \at$. 

\emph{Atomic polynomial forcing} is the statement that $1 \ot f$ and $\mi{\at}(f) \ot 1$ are equal modulo lower terms in $B_\at$, for any $f \in R^J$. Williamson has proven for SW-realizations \cite[Theorem 6.4]{SingSb} that  singular Soergel bimodules have ``standard filtrations,'' which implies atomic polynomial forcing for SW-realizations.  \emph{Polynomial forcing} (Definition \ref{def:generalatomic}) is a generalization of atomic polynomial forcing for a bimodule $B_p$ associated with an arbitrary double coset $p$. 
Using \cite{KELP2, KELP3} we prove Theorem \ref{thm:polyforcinggeneral}, which says that polynomial forcing for any double coset is a consequence of atomic polynomial forcing.

The restrictions imposed on SW-realizations are significant as they rule out some examples of great importance to modular representation theory, e.g. affine Weyl groups in finite characteristic. An \emph{almost SW-realization} (Definition \ref{defn:almostSW}) is a realization over a domain $\Bbbk$ such that one obtains an SW-realization after base change to the fraction field of $\Bbbk$. An example is $\Z[x_1, \ldots, x_n]$ with its action of $\symm_n$. The second main result of this paper is the following equivalence (\Cref{thm.AL=PF}):
\begin{lettertheorem} \label{thm:al=pf}
Let $(W,S)$ be a Coxeter system equipped with an almost SW-realization. Then the atomic Leibniz rule is equivalent to atomic polynomial forcing. 
\end{lettertheorem}
We deduce the SW-realization case of Theorem \ref{thm:alr} from Theorem \ref{thm:al=pf} and Williamson's theory of standard filtrations.

\subsection{}  Let us now explain the motivation for our results.

The diagrammatic presentation for the category of Soergel bimodules, developed by Elias, Khovanov, Libedinsky, and Williamson \cite{Libpresentation},\cite{KhovanovElias}, \cite{Bendihedral}, \cite{Soergelcalculus} has proven to be an important tool for both abstract and computational reasons. This diagrammatically constructed category is typically called the \emph{Hecke category}, and is equivalent to the category of Soergel bimodules for SW-realizations. For non-SW-realizations the category of Soergel bimodules need not behave well, but the diagrammatic category does behave well, and continues to provide a categorification of the Hecke algebra. The ability to compute within the Hecke category has led to advances such as \cite{EWhodge}, \cite{Relativehl},\cite{Geordiesc},\cite{GeordieSimon},\cite{LusWil},\cite{AMRW}. 

The authors are in the midst of a concerted effort to define the \emph{singular Hecke category}, a diagrammatic presentation of singular Soergel bimodules. The framework for such a presentation was developed in \cite{ESW}, but this framework lacks some of the relations needed.  In \cite{KELP4} we described what should be the basis for morphisms in the diagrammatic category, called the \emph{double leaves basis}, and proved that it descends to a basis for morphisms between singular Soergel bimodules for SW-realizations.

Once the remaining relations are understood, proving that the diagrammatic category is correctly presented reduces to proving that double leaves span. An important part of that proof will be \emph{diagrammatic polynomial forcing}, which is like polynomial forcing except that the submodule of lower terms is replaced by the span of double leaves associated to smaller elements in the Bruhat order. We abbreviate diagrammatic polynomial forcing to \emph{$\PF$-forcing}. It is $\PF$-forcing which is our true goal in this paper.

For an atomic coset $\at$ one can identify $\End(B_{\at})$ with $B_{\at}$ as bimodules. In Section \ref{subsect:eval}, we explicitly compute the submodule $\PF_{<\at} \subset B_\at$ corresponding to the $\Bbbk$-linear
subspace of $\mathrm{End}(B_\at)$ spanned by double leaves factoring through
$q < \at$. The following theorem (Corollary \ref{cor.ST=LT} and Theorem \ref{thm.AL=PF})
is proven by direct computation and is the reason we discovered the atomic Leibniz rule in the first place. 

\begin{lettertheorem} \label{thm:al=DLpf}
Let $(W,S)$ be a Coxeter system equipped with an almost SW-realization. Then the atomic Leibniz rule is equivalent to atomic $\PF$-forcing. 
\end{lettertheorem}

In other words, for a given $f \in R^J$, we prove that $1 \ot f - \mi{\at}(f) \ot 1 \in \PF_{<\at}$  if and only if the atomic Leibniz rule holds for $f$ and any $g \in R^J$.

For an SW-realization, the results of \cite{KELP4} can be used to prove that $\PF_{<\at}$ agrees precisely with the submodule of lower terms (see Corollary \ref{cor:pfisltoverSW}). In \S\ref{ssec:localizationtricks} we use some novel localization tricks to extend this statement to almost-SW-realizations. Thus $\PF$-forcing agrees with polynomial forcing for almost-SW-realizations, thus Theorem \ref{thm:al=DLpf} implies Theorem \ref{thm:al=pf}. 

\subsection{}\label{ssec:lastpartofintro}

One is still motivated to prove $\PF$-forcing (or equivalently, the atomic Leibniz rule) for almost-SW-realizations and more general realizations. We now discuss how this process is simplified by Theorem \ref{thm:al=pf}.

Let us say the atomic Leibniz rule \emph{holds for $f \in R^J$} if \eqref{atomicprototype} holds for that specific $f$ and all $g \in R^J$. It is obvious that if the rule holds for $f_1$ and $f_2$ then it holds for $f_1 + f_2$. It is not obvious that if it holds for $f_1$ and $f_2$ then it holds for $f_1 \cdot f_2$, thus the formula \eqref{atomicprototype} is not obviously \emph{multiplicative} in $f$. However, atomic polynomial forcing is obviously multiplicative in $f$. Our equivalence in Theorem \ref{thm:al=pf} is proven on an element-by-element basis, from which we conclude that the atomic Leibniz rule is actually multiplicative.


As a consequence, the atomic Leibniz rule can be proven without relying on Williamson's results, by checking it on each generator $f$ of $R^J$. We perform this computation in type $A$ in \S\ref{s.A}, for carefully chosen generators $f$, via a direct and elementary proof. The operators $T_q$ are simplified dramatically when applied to these generators.

On the one hand, this proves Theorem \ref{thm:alr} for the symmetric group over $\mathbb{Z}[x_1,\ldots, x_n]$, and consequently for other realizations (see Remark \ref{rem:alrawesome}). Via Theorem \ref{thm:al=DLpf}, this implies $\PF$-forcing in the same generality. On the other hand, it also gives a computationally effective way to apply the atomic Leibniz rule (or polynomial forcing), even if one does not know the operators $T_q$ in general: decompose $f$ as a linear combination of products of generators, and apply the atomic Leibniz rule for generators one term at a time.



{\bf Acknowledgments.}  BE was partially supported by NSF grants DMS-2201387 and DMS-2039316. HK was partially supported by the Swedish Research Council. NL was partially supported by FONDECYT-ANID grant 1230247.



\section{Examples and remarks}\label{subsec:examplesintro}

In this chapter we give some examples of atomic Leibniz rules, in the relatively small Coxeter group $W = \symm_4$. The reader should not be disheartened if these examples are still quite difficult and technical to verify by hand! Indeed, one purpose of this chapter is to showcase the complexity involved in even the simplest examples of the atomic Leibniz rule.

The reader will not miss out on anything important if they skip directly to \Cref{s.AL}. 

\subsection{Examples}

Our examples take place in $W = \symm_4$. We let $s = (12)$, $t = (23)$, and $u = (34)$ denote the simple reflections. We give the examples without justification first, and then discuss the verification afterwards. We require that $T_q(f) \in R^K$ for some $K \subset S$ depending on $q$ (precisely: $K = \mi{q}^{-1} I \mi{q} \cap J$).

\begin{ex} \label{ex:S2S2} There are three $(su,su)$-cosets in $\symm_4$: the maximal one $p$ which is atomic, the submaximal coset $q$ with $\mi{q} = t$, and the minimal coset $r$ containing the identity. We claim that for $f, g \in R^{su}$ we have 
	\begin{equation} \label{tsuteq} \pa_{tsut}(fg) = tsut(f) \cdot \pa_{tsut}(g) + \pa_{sut}(T_q(f) \cdot g) + T_r(f) \cdot g, \end{equation}
	or in other words
	\begin{equation} \pa_p(fg) = \mi{p}(f) \cdot \pa_p(g) + \pa_q(T_q(f) \cdot g) + \pa_r(T_r(f) \cdot g), \end{equation}
	where
	\begin{equation} T_q(f) = su\pa_t(f), \qquad T_r(f) = \pa_{tsut}(f) - \pa_{sut}(T_q(f)). \end{equation}
	Note that (obviously) $T_q(f) \in R = R^{\mt}$ while (less obviously) $T_r(f) \in R^{su}$. These are the invariance requirements.
	
	
\end{ex}

\begin{rem} Iterating the ordinary twisted Leibniz rule, it is not hard to deduce the existence of a formula of the form
	\[ \pa_{tsut}(fg) = tsut(f) \cdot \pa_{tsut}(g) + \sum_{x < tsut} \pa_x(T_x(f) g) \]
	for some operators $T_x \co R \to R$. What is not obvious is that, when $f, g \in R^{su}$, this formula will simplify so that only the terms where $x = sut$ and $x=1$ survive. \end{rem}

\begin{ex} \label{ex:S3S1} There are two $(tu,st)$-cosets in $\symm_4$: the maximal coset $p$ which is atomic, and the minimal coset $q$ containing the identity element. We claim that for $f, g \in R^{st}$ we have
	\begin{equation}\label{stueq} \pa_{stu}(fg) = stu(f) \pa_{stu}(g) + \pa_{tu}(T_q(f) \cdot g), \end{equation}
	or in other words
	\begin{equation} \pa_p(fg) = \mi{p}(f) \pa_p(g) + \pa_q(T_q(f) \cdot g), \end{equation}
	where
	\begin{equation} T_q(f) = st\pa_u(f). \end{equation}
	Note also that $T_q(f) \in R^{t}$.
\end{ex}

It helps to look at an example of a non-atomic coset, to see that Leibniz rules are not guaranteed.

\begin{ex} \label{ex:leibnizforq} We return to the notation of Example \ref{ex:S2S2}. Note that $r$ is the only coset less than $q$, and $\pa_r$ is the identity map. Let us argue that there is no naive analogue of \eqref{atomicprototype} for $\pa_q$. If there were, it would have to have the form
	\begin{equation} \label{wontwork} \pa_{sut}(f \cdot g) = sut(f) \pa_{sut}(g) + T(f) \cdot g \end{equation}
	for some operator $T$. By evaluating at $g=1$ we have $T(f) = \pa_{sut}(f)$.
	
	Note that $f, g \in R^{su}$. Iterating the twisted Leibniz rule we have
	\begin{align} \label{thetruth} & \quad \pa_{sut}(f \cdot g) = \\ & sut(f) \pa_{sut}(g) + \pa_s(ut(f)) \pa_{ut}(g) + \pa_u(st(f)) \pa_{st}(g) + \pa_{su}(t(f)) \pa_t(g) + \pa_{sut}(f) g. \nonumber \end{align}
	Only two of these five terms are in \eqref{wontwork}. If $g$ is linear and $\pa_t(g) = 1$ then the sum of the three missing terms is nonzero for some $f$. Thus \eqref{wontwork} is false.
	
	However, now suppose that $f \in t(R^{su})$ instead of $R^{su}$. Then many terms in \eqref{thetruth} vanish, yielding
	\begin{equation} \label{willwork} \pa_{sut}(f \cdot g) = sut(f) \pa_{sut}(g) + \pa_{sut}(f) g. \end{equation}
	This is compatible with \eqref{wontwork}, with $T = \pa_{sut}$.
	
	In view of this example, there is hope of finding some generalization of the atomic Leibniz rule to some non-atomic $(I,J)$ cosets $q$, letting $f \in \mi{q}^{-1}(R^I)$. Be warned that the coset $q$ considered in this example has various special properties, see Remark \ref{rmk:potentialgeneralization}.
\end{ex}



The examples above can easily be verified by computer (all the equations are $R^{stu}$-linear so one need only check the result for $f, g$ in a basis). They can also all be verified using \eqref{demazuredefn} and \eqref{eq:twistedLeibniz}, but this is a very tricky exercise. Doing this exercise may be very instructive for the reader, and emphasizes the difficult and subtlety in these formulae, so we encourage it, and provide some helping hands.

We begin with a few helpful general properties of Demazure operators, which hold in general when $m_{st} = 3$ and $m_{su} = 2$.

\begin{lem} We have
	\begin{equation} \label{quadsign} \pa_s(s(f)) = - \pa_s(f), \quad s \pa_s(f) = \pa_s(f), \qquad \pa_s(\pa_s(f)) = 0,\end{equation}
	\begin{equation} \label{eq:braiddem} st\pa_s(f) = \pa_t(stf), \quad s \pa_t(sf) = t \pa_s(tf), \quad s \pa_u(f) = \pa_u(sf), \end{equation}
	\begin{equation} \alpha_s \pa_s(f) = f - sf, \end{equation}
	\begin{equation} \label{tricky} \pa_{st}(sf) + \pa_{ts}(f) = t \pa_{st}(f). \end{equation}
\end{lem}

Note that $\alpha_{s_i} = x_i - x_{i+1}$ above.

\begin{proof} The reader can verify these relations directly from \eqref{demazuredefn}.  \end{proof}

Most of these formulae are well-known. Meanwhile, we have not seen \eqref{tricky} before; we only use it in Example \ref{ex:twistedleibnizsts} below. With these relations in hand, one need not refer to the original definition \eqref{demazuredefn} again, and need only use the twisted Leibniz rule.

Here are some example computations using \eqref{quadsign} and \eqref{eq:braiddem}.
\begin{equation} \pa_s(t \pa_s(f)) = - \pa_s(st\pa_s(f)) = - \pa_s \pa_t(stf) = \pa_s \pa_t(tstf). \end{equation}
\begin{equation} \pa_s(t \pa_s(f)) = \pa_s(ts \pa_s(f)) = ts\pa_t\pa_s(f). \end{equation}

Let us consider example \ref{ex:S3S1}. It helps to observe that
\begin{equation} \pa_{st}(u(f)) \in R^{st}, \qquad \pa_{stu}(f) \in R^{stu}, \end{equation}
under the assumption that $f \in R^{st}$. One way to see the equality \eqref{stueq} is to expand both sides using the twisted Leibniz rule. Since $g \in R^{st}$, any terms containing $\pa_{st}(g)$ or $\pa_s(g)$ or $\pa_t(g)$ or $\pa_{su}(g)$ are zero. Now compare the coefficients of $g$, $\pa_u(g)$ and $\pa_{tu}(g)$ respectively. On the left the coefficient of $g$ is $\pa_{stu}(f)$, while on the right the coefficient of $g$ is  \[\pa_{tu}st \pa_u(f)=\pa_ts\pa_u tu\pa_u(f)=\pa_tstu\pa_{tu}(f)=st\pa_s u\pa_{tu}(f)=stu\pa_{stu}(f)=\pa_{stu}(f).\]
We have applied \eqref{eq:braiddem} repeatedly and used that $\pa_{stu}(f) \in R^{stu}$. Thus the coefficients of $g$ are the same on both sides of \eqref{stueq}. We leave the other coefficients to the reader.

Example \ref{ex:S2S2} is the hardest. One should first prove the following statements under the critical assumption that $f, g \in R^{su}$.
\begin{subequations}
	\begin{equation} \pa_{tsut}(f) \in R^{stu}, \qquad  \pa_{ts}(ut(f)) \in R^{st},\end{equation}
	\begin{equation} \pa_{su}(tsu(\pa_t(f))) = tsut(\pa_{sut}(f)), \end{equation}
	\begin{equation} t(\pa_{sut}(f)) = tsu(\pa_{sut}(f)),
	\end{equation}
	\begin{equation} tsu\left[\pa_{sut}(f) - t \pa_{sut}(f)\right] = tsu(\alpha_t \pa_{tust}(f)) = (\alpha_s + \alpha_t + \alpha_u) \pa_{tsut}(f), \end{equation}
	\begin{equation} \pa_{tsu}(tf) = \pa_{tsu}(-\alpha_t \pa_t(f)) = -(\alpha_s + \alpha_t + \alpha_u) \pa_{tsut}(f). \end{equation}
\end{subequations}
We leave these verifications, and the deduction of \eqref{tsuteq} therefrom, to the ambitious reader.

\subsection{Leibniz rules for permutations} \label{ssec:leibnizpermutations}

A natural question is raised: for which double cosets $p$ does one expect an equality of the form \eqref{atomicprototype} to hold? Given the discussion in Example \ref{ex:leibnizforq}, one might ask instead: for which double cosets $q$ does \eqref{atomicprototype} hold under the alternate assumption that $f \in \mi{q}^{-1}(R^I)$? Note that $f \in \mi{q}^{-1}(R^I)$ is equivalent to $f \in R^J$ for atomic cosets, and also for more general cosets called \emph{core cosets}. We believe these are interesting questions, even though these more general Leibniz rules currently lack a clear connection to the theory of singular Soergel bimodules. 

A special case would be when $I = J = \mt$, so that double cosets are in bijection with elements $w \in W$; all such cosets are core. 
We start with the following well-known lemma. 

\begin{lem} \label{lem:iteratedleibnizforw}
	Let $w\in W$ 
 and $f,g \in R$. Let $w=s_1\ldots s_n$ be a reduced expression and for $\eb=\{0,1\}^n$ we define the element $w^{\eb}=s_1^{e_1}\ldots s_n^{e_n}$. For $\eb=\{0,1\}^n$ let 
	\[ \theta_i^{\eb}=\begin{cases} s_i&\text{if }e_i=1\\
		\partial_i&\text{if }e_i=0 \end{cases} \text{ and }
	\Theta^{\eb}(f)=\theta_1^{\eb}\circ \theta_2^{\eb}\circ \ldots \circ \theta_n^{\eb}(f).\]
	
	Then we have
	\[ \pa_w(fg) = \sum_{x \leq w} T'_x(f) \pa_x(g)\]
	where 
	\[ T'_x(f)=\sum_{\eb \mid w^{\eb} =x} \Theta^{\eb}(f).\]
As a special case, $T_w'(f) = w(f)$.
\end{lem}

\begin{proof} This is just an iteration of the twisted Leibniz rule. \end{proof}

To summarize, we obtain a generalized Leibniz rule of the form
\begin{equation} \label{version1intro} \pa_w(fg) = w(f) \pa_w(g) + \sum_{x < w} T'_x(f) \pa_x(g), \end{equation}
for operators $T'_x \co R \to R$ (depending on $w$). The \emph{formula} for $T'_x$ one derives in this way is seemingly dependent on the choice of reduced expression for $w$, though the \emph{operator} only depends\footnote{Abstractly, the nilHecke algebra is the subalgebra of $\End(R)$ generated by $R$ (i.e. multiplication by polynomials) and by Demazure operators. It is well known that the operators $\{\pa_w\}$ form a basis of the nilHecke algebra as a free left $R$-module. Letting $m_f$ denote multiplication by $f$,  \eqref{version1intro} can be viewed as an equality
\[ \pa_w \circ m_f = m_{w(f)} \circ \pa_w + \sum m_{T'_x(f)} \circ \pa_x \]
in the nilHecke algebra, which rewrites $\pa_w \circ m_f$ in this basis. Consequently, the coefficients $T'_x(f)$ in this linear combination depend only on $w$ and $f$.} on $w$. In practice, confirming the independence of reduced expression can be quite subtle. We are unaware of a formula for $T'_x$ which is obviously independent of the choice of reduced expression.

Meanwhile, one can also deduce an equality of the form
\begin{equation} \label{version2intro} \pa_w(fg) = w(f) \pa_w(g) + \sum_{x < w} \pa_x(T_x(f) \cdot g), \end{equation}
for some operators $T_x \co R \to R$. We are unaware of any previous study of the operators $T_x$ and the formula \eqref{version2intro}.

\begin{rem} Later in the paper, we also discuss an atomic Leibniz rule similar to \eqref{version1intro} rather than \eqref{version2intro}. \end{rem}

Here are some examples.

\begin{ex} Let $s = s_1$ and $t = s_2$. When $w = ts$, applying \eqref{eq:twistedLeibniz} twice gives
	\begin{equation} \label{tsfirst} \pa_{ts}(fg) = ts(f) \pa_{ts}(g) + \pa_t(sf) \pa_s(g) + t\pa_s(f) \pa_t(g) + \pa_{ts}(f) g. \end{equation}
	An equivalent formula is
	\begin{equation}\label{ts} \pa_{ts}(fg) = ts(f) \pa_{ts}(g) + \pa_t(\pa_s(f) \cdot g) + \pa_s(s\pa_t(sf) \cdot g) + \pa_{st}(sf) \cdot g. \end{equation}
	By applying \eqref{eq:twistedLeibniz} to the second and third terms on the RHS of equation \eqref{ts}, and with a little help from \eqref{quadsign}, one obtains \eqref{tsfirst}.
\end{ex}

\begin{ex} \label{ex:twistedleibnizsts} With notation as above,  when $w = sts$, applying  \eqref{eq:twistedLeibniz} thrice gives
	\begin{eqnarray} \label{version1sts} \pa_{sts}(fg) &= & sts(f) \pa_{sts}(g) + \pa_s(tsf) \pa_{ts}(g) + \pa_t(stf) \pa_{st}(g) \\ \nonumber &&+ \pa_s(t\pa_s(f)) \pa_t(g) + \pa_t(s\pa_t(f)) \pa_s(g) + \pa_{sts}(f) g.\end{eqnarray}
	More honestly, applying \eqref{eq:twistedLeibniz} thrice gives the above except that the coefficient of $\pa_s(g)$ is $\pa_{st}(sf) + s \pa_{ts}(f)$. By applying $s$ to \eqref{tricky}, one obtains 
	\begin{equation} \pa_{st}(sf) + s \pa_{ts}(f) = st \pa_{st}(f) = \pa_t(st \pa_t(f)) = \pa_t (s \pa_t(f)). \end{equation}
	This is how one deduces \eqref{version1sts}.
	
	An equivalent formula is
	\begin{eqnarray} \label{version2sts} \pa_{sts}(fg) &=& sts(f) \pa_{sts}(g) + \pa_{ts}(\pa_t(f) \cdot g) + \pa_{st}(\pa_s(f) \cdot g)\\ \nonumber && - \pa_t(\pa_{st}(f) \cdot g) - \pa_s(\pa_{ts}(f) \cdot g) + \pa_{sts}(f) \cdot g. \end{eqnarray}
	To verify that \eqref{version2sts} and \eqref{version1sts} agree, apply \eqref{ts} to the term $\pa_{ts}(\pa_t(f) \cdot g)$, and apply \eqref{ts} with $s$ and $t$ swapped to $\pa_{st}(\pa_s(f) \cdot g)$. After some additional massaging using \eqref{quadsign} and \eqref{eq:braiddem}, one will recover \eqref{version1sts}. 
	

\end{ex}

\section{Atomic Leibniz rules}\label{s.AL}

\subsection{Realizations}\label{ss.realizations}
 
We fix a Coxeter system $(W,S)$. We let $e$ denote the identity element of $W$. Recall the definition of a realization from \cite[\S 3.1]{Soergelcalculus}.

\begin{defn} A \emph{realization} of $(W,S)$ over $\Bbbk$ is
the data $(\Bbbk,V,\Delta, \Delta^\vee)$ of a commutative ring $\Bbbk$, a free finite-rank $\Bbbk$-module $V$, a set $\Delta = \{\alpha_s\}_{s \in S}$ of \emph{simple roots} inside $V$, and a set $\Delta^\vee = \{\alpha_s^\vee\}_{s \in S}$ of \emph{simple coroots} inside $\Hom_{\Bbbk}(V,\Bbbk)$, satisfying the following properties. One has $\alpha_s^\vee(\alpha_s) = 2$ for all $s \in S$. The formula
\[ s(v) := v - \alpha_s^\vee(v) \cdot \alpha_s \]
defines an action of $W$ on $V$. Also, the technical condition \cite[(3.3)]{Soergelcalculus} holds, which is redundant for most base rings $\Bbbk$. \end{defn}

For short, we often refer to the data of a realization simply by reference to the $W$-representation $V$.

\begin{ex} \label{ex:permrealization} The \emph{permutation realization} of $\symm_n$ over $\Z$ has $V = \Z^n$ with basis $\{x_i\}_{i=1}^n$ and with dual bases $\{x_i^*\}$. For $1 \le i \le n-1$ one sets $\alpha_i = x_i - x_{i+1}$ and $\alpha_i^\vee = x_i^* - x_{i+1}^*$.\end{ex}

\begin{ex} For a Weyl group $W$, the \emph{root realization} of $(W,S)$ over $\Z$ is the free $\Z$-module with basis $\Delta$. One defines $\Delta^\vee$ so that the pairings $\alpha_s^\vee(\alpha_t)$ agree with the usual Cartan matrix of $W$. \end{ex}

\begin{ex} Let $(\Bbbk,V,\Delta,\Delta^\vee)$ be a realization of $(W,S)$, and $I \subset S$. Then $(\Bbbk,V,\Delta_I,\Delta_I^\vee)$ is also a realization of $(W_I,I)$, the \emph{restriction of the realization to a parabolic subgroup}. Here $\Delta_I = \{\alpha_s\}_{s \in I}$, and similarly for $\Delta_I^{\vee}$. \end{ex}

Given a realization, let $R$ be the polynomial ring whose linear terms are $V$. We can associate Demazure operators $\pa_s \colon R \to R$ for $s \in S$, which agree with $\alpha_s^\vee$ on $V \subset R$, and are extended by the twisted Leibniz rule. For details, see \cite[\S 3.1]{Soergelcalculus}.
For each finitary subset $I \subset S$, we also consider the subring $R^I$ of $W_I$-invariants in $R$. The ring $R^I$ is graded, and all $R^I$-modules will be graded, but we will not keep track of grading shifts in this paper as they will play no significant role. The background on this material in \cite{KELP1} should be sufficient.

\begin{defn} \label{defn:frobrealization}  A \emph{(balanced) Frobenius realization} is a realization satisfying the following properties, see \cite[\S 3.1]{KELP1} for definitions. \begin{itemize}
\item It is balanced.
\item It satisfies generalized Demazure surjectivity.
\item It is faithful when restricted to each finite parabolic subgroup $W_I$.
\end{itemize}
\end{defn}

We assume tacitly throughout this paper that we work with a Frobenius realization. The main implication of these assumptions is that, when $I \subset S$ is finitary, the Demazure operator $\pa_{w_I} \co R \to R^I$ is well-defined and equips the ring extension $R^I \subset R$ with the structure of a Frobenius extension.

The left and right redundancy sets of an $(I,J)$-coset $p$ are defined and denoted as
\begin{equation} \leftred(p) = I \cap \mi{p} J \mi{p}^{-1}, \qquad \rightred(p) = \mi{p}^{-1} I \mi{p} \cap J. \end{equation}
An $(I,J)$-coset $p$ is a \emph{core coset} if $I = \leftred(p)$ and $J = \rightred(p)$.

For any $(I,J)$-coset $q$, in \cite{KELP1} we define a Demazure operator \[ \pa_q : R^J \to R^I.\] By definition, $\pa_q$ is the restriction of the ordinary Demazure operator 
\[ \pa_{\ma{q} w_J^{-1}} \co R \to R\] to the subring $R^J$. After restriction, the image is contained in $R^I$, see \cite[Lemma 3.9]{KELP1}. Note that $\ma{q} w_J^{-1} = \mi{q}$ if and only if $q$ is a core coset.

\begin{rem} 
Some results from \cite{KELP1} require further that the realization is faithful, rather than just faithful upon restriction to each finite parabolic subgroup. In particular, the set $\{\pa_p\}$ as $p$ ranges over $(I,J)$-cosets need not be linearly independent when the realization is not faithful. \end{rem}

A \emph{multistep $(I,J)$-expression} is a sequence of finitary subsets 
\[ I_{\bullet} = [[I= I_0 \subset K_1 \supset I_1 \subset \ldots K_m \supset I_m = J]]. \] The definition of a \emph{reduced} multistep expression, and of the $(I,J)$-coset that it \emph{expresses}, can be found in \cite[Definition 1.4]{EKo}. When $I_{\bullet}$ is a (reduced) expression which expresses $p$, we write $p \expr I_{\bullet}$.

As in \cite{EKo}, for $x, y \in W$ we write $x.y$ for a reduced composition, where $\ell(xy) = \ell(x) + \ell(y)$. We also use this notation for the reduced composition of reduced expressions, or the reduced composition of double cosets, see \cite{EKo} for more details. Demazure operators compose well over reduced compositions: one has $\pa_{p.q} = \pa_p \circ \pa_q$, as proven in \cite[Corollary 3.19]{KELP1}.

\subsection{Precise statement of atomic Leibniz rules}

\begin{rem} Note that $w_J$ is an involution, so $w_J = w_J^{-1}$. We write $\ma{q} w_J^{-1}$ above to emphasize that $\ma{q} = (\ma{q} w_J^{-1}) . w_J$. \end{rem}


\begin{defn} \label{defn:atomicLeibniz} Suppose $M$ is finitary, $s \in M$, and $t = w_M s w_M$. Let $I = \hat{s} := M \setminus s$, and $J = \hat{t} := M \setminus t$. Let $\at$ be the (atomic) $(I,J)$-coset containing $w_M$.  We say a \emph{(rightward) atomic Leibniz rule} holds for $\at$ if there exist $R^M$-linear operators $T^\at_q$ from $R^{J}$ to $R^{\rightred(q)}$ for each $(I,J)$-coset $q < \at$, such that for any $f, g \in R^{J}$ we have
	\begin{equation} \label{eq:atomicleibniz} \pa_\at(f \cdot g) = \mi{\at}(f) \pa_\at(g) + \sum_{q < \at} \pa_{\ma{q} w_J^{-1}}(T^\at_q(f) \cdot g). \end{equation}
\end{defn}

We encourage the reader to confirm that $T^\at_q(f) \in R^{\rightred(q)}$ in the examples of \S\ref{subsec:examplesintro}. We continue to write $T_q$ instead of $T^\at_q$ when $\at$ is understood.

\begin{rem} We say ``an atomic Leibniz rule'' rather than ``the atomic Leibniz rule'' because we are defining a prototype for a kind of formula.  If one specifies operators $T_q$ such that the formula holds, then one has produced ``the'' atomic Leibniz rule for that coset $\at$ (indeed, we prove in \Cref{thm.AL=PF} that such operators are unique for certain realizations). \end{rem}

The difference between \eqref{eq:atomicleibniz} and \eqref{atomicprototype} is subtle: we have written $\pa_{\ma{q} w_J^{-1}}$ instead of $\pa_q$. The difference between $\pa_{\ma{q} w_J^{-1}}$ and $\pa_q$ is only a matter of the domain and codomain of the functions: the former is a function $R \to R$, while the latter is its restriction to a function $R^J \to R^I$. Meanwhile, $T_q(f)$ lives in $R^{\rightred(q)}$. The inclusion $R^J \subset R^{\rightred(q)}$ is proper unless $q$ is core. It is therefore inappropriate to apply $\pa_q$ to $T_q(f) \cdot g$. Having altered notation so that the domain of the operator is appropriate, we still need to worry about the codomain, which we address in the following lemma.

\begin{lem} With notation as in Definition \ref{defn:atomicLeibniz}, we have $\pa_{\ma{q} w_J^{-1}}(T_q(f) \cdot g) \in R^I$. \end{lem}

\begin{proof}
	Recall from \cite[Proposition 4.28]{EKo} that any $(I,J)$-coset $q$ has a reduced expression of the form
	\begin{equation}\label{rexforq}  q \expr [[I \supset \leftred(q)]] . q^{\core} . [[\rightred(q) \subset J]]. \end{equation}
	Let $z$ be the $(I,\rightred(q))$-coset with reduced expression
	\begin{equation}\label{eq:z} z \expr [[I \supset \leftred(q)]] . q^{\core}. \end{equation}
	Since \eqref{rexforq} is reduced, by \cite[Proposition 4.3]{EKo}, we have $	 \ma{q} = \ma{z}.(w_{\rightred(q)}^{-1} w_J),$ so that \begin{equation}\label{eq:qj}\ma{q} w_J^{-1} = \ma{z} w_{\rightred(q)}^{-1}.\end{equation} Consequently, the same operator $\pa_{\ma{q} w_J^{-1}} \co R \to R$ restricts to both $\pa_q \co R^J \to R^I$ and $\pa_z \co R^{\rightred(q)} \to R^I$. In particular, this operator sends $R^{\rightred(q)}$ to $R^I$. \end{proof}

Further elaboration will be helpful in subsequent chapters. By \cite[Lemma 3.17]{KELP1}, the reduced expression \eqref{rexforq} implies that the map
$\pa_q$ is a composition of three Demazure operators.  Recall that $\pa_{[[I \supset \leftred(q)]]}$ is the Frobenius trace map $R^{\leftred(q)} \to R^I$ often denoted as $\pa^{\leftred(q)}_I$. Recall also that $\pa_{[[\rightred(q) \subset J]]}$ is the inclusion map $R^J \subset R^{\rightred(q)}$. We denote this inclusion map $\iota^{\rightred(q)}_J$ below. So we have
\begin{equation} \label{eq:factorpaq}  \pa_q = \pa^{\leftred(q)}_{I} \circ \pa_{q^{\core}} \circ \iota^{\rightred(q)}_J. \end{equation}
By \eqref{eq:z} we have
\begin{equation} 
\pa_z= \pa^{\leftred(q)}_{I} \circ \pa_{q^{\core}},\end{equation}
which agrees with the restriction of $\pa_{\ma{q} w_J^{-1}}$ to $R^{\rightred(q)}$.
Thus  one has the following reformulation of \eqref{eq:atomicleibniz}:

\begin{equation} \label{eq:atomicleibnizBETTER} \pa_\at(f \cdot g) = \mi{\at}(f) \pa_\at(g) + \sum_{q < \at} \pa^{\leftred(q)}_I \pa_{q^{\core}} (T_q(f) \cdot \iota^{\rightred(q)}_{J}(g)). \end{equation}
Now the polynomial $T_q(f)$ appears more appropriately in the ``middle'' of this factorization of $\pa_q$. This discussion of the ``placement'' of the polynomial $T_q(f)$ will play a role in our diagrammatic proof of polynomial forcing.

We are now prepared to discuss another version of the atomic Leibniz rule, using the factorization \eqref{eq:factorpaq}. It should not be obvious that these two atomic Leibniz rules are related, though the equivalence with polynomial forcing will shed light on this issue.

\begin{defn} \label{defn:atomicLeibnizALT} Use the notation from Definition \ref{defn:atomicLeibniz} and from \eqref{eq:factorpaq}. We say that a \emph{leftward atomic Leibniz rule} holds for $\at$ if there exist $R^M$-linear operators $T'_q$ from $R^J$ to $R^{\leftred(q)}$ for each $(I,J)$-coset $q < \at$, such that for any $f, g \in R^{J}$ we have
\begin{equation} \label{eq:atomicleibnizALT} \pa_\at(f \cdot g) = \mi{\at}(f) \pa_\at(g) + \sum_{q < \at} \pa^{\leftred(q)}_{I}(T'_q(f) \cdot \pa_{q^{\core}}(\iota_{J}^{\rightred(q)}(g))). \end{equation}
\end{defn}

\begin{rem} \label{rmk:potentialgeneralization} The fact that $T_q(f)$ lives in $R^{\rightred(q)}$ and not in $R^J$ is easy to overlook, but overlooking it is dangerous. We have attempted to prove atomic Leibniz-style rules for more general families of cosets (core cosets, cosets whose core is atomic, etcetera). Each time what prevents one from bootstrapping from the atomic case to more general cases is the fact that $T_q(f)$ does not live in $R^J$. The generalization in Example \ref{ex:leibnizforq} has the special feature that the lower cosets are all core, so that their right redundancy equals $J$. (It also has the special feature that $q^{\core}$ is atomic.) \end{rem}

\subsection{Changing the realization} \label{ssec:changerealization}

We argue that the atomic Leibniz rule for some realizations implies the atomic Leibniz rule for others. Given a realization, one can obtain another realization by applying base change $(-) \ot_{\Bbbk} \Bbbk'$ to $V$, and choosing new roots and coroots in the natural way. We call this a \emph{specialization}. Here are two other common ways to alter the realization.

\begin{defn} Let $(V,\Delta,\Delta^\vee)$ be a realization of $(W,S)$ over $\Bbbk$. Let $N$ be a free $\Bbbk$-module acted on trivially by $W$. Then $(V \oplus N, \Delta \oplus 0,\Delta^\vee \oplus 0)$ is a realization, called a \emph{$W$-invariant enlargement} of the original. More precisely, the new roots are the image of the old roots under the inclusion map, and the new coroots kill the summand $N$. \end{defn}

\begin{defn} Let $(V,\Delta,\Delta^\vee)$ be a realization of $(W,S)$ over $\Bbbk$. Suppose one has a decomposition $V = X \oplus Y$ of free $\Bbbk$-modules, such that $W$ acts trivially on $X$, though $W$ need not preserve $Y$. Note that the coroots necessarily annihilate $X$.  Then $(Y, \overline{\Delta},\Delta^\vee_Y)$ is a realization, called a \emph{$W$-invariant quotient} of the original. Here, we identify $Y$ as the quotient $V/X$, and $\overline{\Delta}$ represents the image of $\Delta$ under the quotient map. The functionals $\Delta^\vee$ kill $X$, so they descend to functionals $\Delta^\vee_Y$ on $Y$.  We also make the technical assumption\footnote{This assumption is required for the $W$-invariant quotient to satisfy Demazure surjectivity.} that $\alpha_s$ induces a surjective map $Y^* \to \Bbbk$.
\end{defn}

\begin{ex} Let $(W,S)$ have type $\widetilde{A}_{n-1}$, with simple reflections $s_i$ for $1 \le i \le n$. Let $V$ be the free $\Z$-module spanned by $\{x_i\}_{i=1}^n$ and $\delta$. Let $\{x_i^*\} \cup \{\delta^*\}$ denote the dual basis in $\Hom_{\Bbbk}(V,\Bbbk)$. With indices considered modulo $n$, let $\alpha_i = x_i - x_{i+1} + \delta$, and $\alpha_i^\vee = x_i^* - x_{i+1}^*$. This is a realization of $(W,S)$ called the \emph{affine permutation realization}. Note that $\sum_{i=0}^{n-1} \alpha_i = n \delta$, which is $W$-invariant. Let $X$ be the span of $\delta$, and $Y$ be the span of $\{x_i\}_{i=1}^n$. Note that $W$ does not preserve $Y$, since the roots are not contained in $Y$. There is a valid $W$-invariant quotient $(Y,\overline{\Delta}, \Delta^\vee_Y)$ which agrees, upon restriction to the parabolic subgroup $\symm_n$ generated by $\{s_i\}_{i=1}^{n-1}$, with the permutation representation. \end{ex}

\begin{ex} \label{ex:restrictingaffineperm} Continuing the previous example, let $y_i = x_i - i \delta$. Then we can also view $V$ as having basis $\{y_i\}_{i=1}^n \cup \{\delta\}$, and $\alpha_i = y_i - y_{i+1}$ for $i \ne n$. Upon restriction to the parabolic subgroup $\symm_n$, we see that $V$ is isomorphic to the $W$-invariant enlargement of the permutation representation of $\symm_n$ (with basis $\{y_i\}$) by the $W$-invariant span of $\delta$.

Indeed, $W$ has $n$ distinct maximal parabolic subgroups isomorphic to $\symm_n$ as groups. A similar construction will show that the restriction of $V$ to any maximal parabolic subgroup (a copy of $\symm_n$) will be isomorphic to an invariant enlargement of its permutation representation. \end{ex}

\begin{lem}\label{lem:changerealization} If a rightward (resp. leftward) atomic Leibniz rule holds for a 
Frobenius realization, then it also holds for specializations, $W$-invariant enlargements, and $W$-invariant quotients. \end{lem}

\begin{proof}  Let $R$ be the ring associated to the original realization, and $R_{\new}$ be the realization associated to the specialization, enlargement, or quotient.
All three cases are united by the fact that $R_{\new}$ is a tensor product of the form $R \ot_{A} B$, where $A \subset R$ is a subring on which $W$ acts trivially, and $B$ is a ring on which $W$ acts trivially. For specializations we have $R_{\new} = R \ot_{\Bbbk} \Bbbk'$; for enlargements we have  $R_{\new} = R \ot_{\Bbbk} R_N$, where $R_N$ is the polynomial ring of $N$; for quotients we have $R_{\new} = R \ot_{R_X} \Bbbk$, where $R_X$ is the polynomial ring of $X$, and $\Bbbk$ is its quotient by the ideal of positive degree elements. For $w \in W$, its action on $R_{\new}$ is given by $w \ot \id$. The roots in $R_{\new}$ are given by $\alpha_s \ot 1$, and the Demazure operators $\pa^{\new}_s$ on $R_{\new}$ have the form $\pa_s \ot \id$.

The important point in all three cases is that for each $I \subset S$ finitary we have $R_{\new}^I = R^I \ot_A B$. We now prove this somewhat subtle point. There is an obvious inclusion $R^I \ot_A B \subset R_{\new}^I$, so we need only show the other inclusion.

It is straightforward to verify that the new realization satisfies generalized Demazure surjectivity.
A consequence is that the typical properties of Demazure operators are satisfied. For example, the kernel and the image of $\partial^{\new}_s$ are both equal to $R_{\new}^s$, and $\partial^{\new}_s$ is $R^s_{\new}$-linear. 
It follows that $\pa^{\new}_I$ is also $R_{\new}^I$-linear.

Suppose that $g \in R_{\new}^I$, and write $g = \sum f_i \ot b_i$. Choose some $P \in R$ with $\pa_I(P) = 1$, which exists by generalized Demazure surjectivity. Then $\pa^{\new}_I(P \ot 1) = 1 \ot 1$ in $R_{\new}$. Thus
\begin{equation} g = g \pa^{\new}_I(P \ot 1) = \pa^{\new}_I(g \cdot (P \ot 1)) = \sum \pa_I(f_i \cdot P) \ot b_i. \end{equation}
Thus $g \in R^I \ot_A B$.

The rest of the proof is straightforward. Fix an atomic coset $\at$. For each $q < \at$, given operators $T_q$ for the original realization satisfying \eqref{eq:atomicleibniz}, we define $T_{q,\new} := T_q \ot \id$. By linearity, we need only check \eqref{eq:atomicleibniz} for $R_{\new}$ on elements in $R_{\new}^J$ of the form $f \ot b_1$ and $g \ot b_2$ for $f, g \in R^J$. It is easy to verify \eqref{eq:atomicleibniz} for $R_{\new}$ on such elements, since all operators (like $\mi{\at}$ or $\pa_{\ma{q} w_J^{-1}}$) are applied only to the first tensor factor, where we can use the atomic Leibniz rule from $R$. We conclude by noting that $T_{q,\new}$ has the appropriate codomain as well.
\end{proof}

We do not claim that any statements about the unicity of the operators $T_q$ will extend from a realization to its specializations, enlargements, or quotients.

\begin{lem} Let $(\Bbbk,V,\Delta,\Delta^\vee)$ be a realization of $(W,S)$. If one can prove an atomic Leibniz rule for the restriction of $V$ to $W_M$, for all (maximal) finitary subsets $M \subset S$, then an atomic Leibniz rule holds for $W$. \end{lem}

\begin{proof} Every atomic coset in $W$ lives within $W_M$ for some finitary $M$ (which lives within a maximal finitary subset), and the same atomic Leibniz rule which works for $W_M$ will work for $W$. \end{proof} 

\begin{ex} \label{ex:atomicforaffine} Suppose one can prove an atomic Leibniz rule for the permutation realization of $\symm_n$ over $\Z$. Then by enlargement, one obtains an atomic Leibniz rule for the affine permutation realization restricted to any finite parabolic subgroup, see Example \ref{ex:restrictingaffineperm}. By the previous lemma, an atomic Leibniz rule holds for the affine permutation realization of the affine Weyl group of type $\widetilde{A}_{n-1}$.  \end{ex}

\section{Lower terms} \label{sec:lowerterms}

In this section we give an explicit description of the ideal of lower terms for an atomic coset using the technology of singular light leaves. 

\subsection{Definition of lower terms}

\begin{defn} Let 
\[ I_{\bullet} = [[I= I_0 \subset K_1 \supset I_1 \subset \ldots K_m \supset I_m = J]]\] be a multistep $(I,J)$ expression. To this expression we associate a \emph{(singular) Bott-Samelson bimodule}
\[ \BS(I_{\bullet}) := R^{I_0} \ot_{R^{K_1}} R^{I_1} \ot_{R^{K_2}} \cdots \ot_{R^{K_m}} R^{I_m}.\]
This is an $(R^I, R^J)$-bimodule.
The collection of all Bott-Samelson bimodules is closed under tensor product, and forms (the set of objects in) a full sub-2-category of the 2-category of bimodules. This sub-2-category is denoted $\SBSBim$.
\end{defn}

For two $(R^I, R^J)$-bimodules $B$ and $B'$, $\Hom(B,B')$ denotes the space of bimodule maps. Moreover, $\Hom(B,B')$ is itself an $(R^I, R^J)$-bimodule in the usual way.

Inside any linear category, given a collection of objects, their identity maps generate a two-sided ideal. This ideal consists of all morphisms which factor through one of those objects, and linear combinations thereof. In the context of $(R^I, R^J)$-bimodules, the actions of $R^I$ and $R^J$ commute with any morphism, and thus preserve the factorization of morphisms. Hence the morphisms within any such ideal form a sub-bimodule of the original Hom space.

\begin{defn}\label{def.lowerterm} Let $p$ be an $(I,J)$-coset. Consider the set of reduced expressions $M_{\bullet}$ for any $(I,J)$-coset $q$ with $q < p$. Let $\Hom_{< p}$ denote the ideal in the category of $(R^I, R^J)$-bimodules 
generated by the identity maps of $\BS(M_{\bullet})$ for such expressions. Then $\Hom_{<p}$ is a two-sided ideal, the \emph{ideal of lower terms} relative to $p$. The ideal $\Hom_{\leq p}$ is defined similarly. \end{defn}

So $\Hom_{< p}(B,B')$ is a subset of $\Hom(B,B')$, and is a sub-bimodule for $(R^I, R^J)$.
We write $\End_{< p}(B)$ instead of $\Hom_{<p}(B,B) \subset \End(B)$.

We now focus on the case of atomic cosets. We use the letter $\at$ to denote an atomic coset and let $[[I\subset M\supset J]]$ denote the unique reduced expression of $\at$. We let
\begin{equation} B_\at := \BS([[I \subset M \supset J]]) = R^I \ot_{R^M} R^J. \end{equation}


Because $B_\at$ is generated by $1 \ot 1$ as a bimodule, any endomorphism is determined by where it sends this element. Thus 
\begin{equation}\label{eq:endatomic} \End(B_\at)\cong R^I \ot_{R^M} R^J 
\end{equation}
as $(R^I, R^J)$-bimodules, via the operations of left and right multiplication. Hence $\End(B_\at) \cong B_\at$ as $(R^I, R^J)$-bimodules\footnote{We have ignored gradings in this paper. Using traditional grading conventions for Bott-Samelson bimodules, $\End(B_\at)$ and $B_\at$ are only isomorphic up to shift. The identity map of $\End(B_\at)$ is in degree zero, while $1 \ot 1 \in B_\at$ is not.}. It is easy to deduce that $B_\at$ is indecomposable (when $\Bbbk$ is a domain) 
since there are no non-trivial idempotents in $\End(B_\at)\cong B_\at$.

\subsection{Atomic double leaves}

The goal of the section is to describe a large family of morphisms in $\End(B_{\at})$ called double leaves, most of which are in $\End_{< \at}(B_{\at})$ by construction. We use the diagrammatic technology originally found in \cite{ESW} and developed further in \cite{KELP4}. 

We assume a Frobenius realization, see Definition \ref{defn:frobrealization}. In particular, the ring inclusions $R^I \subset R^J$ are Frobenius extensions. Under these assumptions, a diagammatic 2-category $\Frob$ is constructed in \cite{ESW}, and it comes equipped with a 2-functor to $\SBSBim$. This 2-functor is essentially surjective, but is not expected to be an equivalence; the category $\Frob$ is missing a number of relations.

Double leaves are to be constructed either as morphisms in $\Frob$, or as their images in $\SBSBim$, depending on the context.

The objects in $\Frob$ are indexed not by multistep expressions but by singlestep expressions. An  \emph{$(I,J)$ singlestep expression} is a sequence
\[ I_{\bullet} = [I = I_0, I_1, \ldots, I_d = J] \]
where each $I_i$ is a finitary subset of $S$, and each $I_i$ and $I_{i+1}$ differ by the addition or removal of a single simple reflection.  
We use single brackets for singlestep expressions, and double brackets for multistep expressions.

Throughout this section, we fix $I\subset M=Is\subset S$ finitary, and let $t = w_M s w_M$ and $J=M\setminus t$, so that 
\[\at\expr[I,M,J]\]
is an atomic coset.
We also fix the $(I,M)$-coset 
\[n=W_IeW_M.\]

\subsubsection{Elementary light leaves for atomic Grassmannian pairs}\label{ss.ELL}

By definition of atomic, $\overline{\at}=w_M$. Then for an $(I,J)$-coset $q$, then condition $q\leq \at$ is equivalent (see \cite[Thm 2.16]{KELP2}) to $\overline{q}\leq w_M$, which in turn is equivalent to $q\subseteq n=W_M.$

For an $(I,J)$-coset $q$ contained in $W_M$, the pair $q\subset n$ is Grassmannian in the sense of \cite[Definition 2.7]{KELP4}. Associated to such a pair, \cite[Section 7.3]{KELP4} constructs a distinguished map called an elementary light leaf. 
The map (and codomain of the map) depends on a choice we make now:
we fix a reduced expression $X_q$ of the form
\begin{equation}\label{eq.Xq}
X_q = [[I\supset \leftred(q)]]\circ X_q^{\core}\circ[[\rightred(q)\subset J]] 
\end{equation}
where $X_q^{\core}$ is a reduced expression of $q^{\core}$.

\begin{defn}\label{def.ell}
Let $q$ be an $(I,J)$-coset contained in $n$. The \emph{elementary light leaf associated to $[n,q]$ (and $X_q$)} is the $(R^I,R^J)$-bimodule morphism
\[\ELL([n,q]):\BS([I,M,J])\to \BS(X_q)\] sending the generator $1\otimes 1\in R^I\otimes_{R^M} R^J=\BS([I,M,J])$ to the element $1^{\otimes}:= 1\otimes \cdots \otimes 1\in \BS(X_q)$. Equivalently, $\ELL([n,q])$ is defined by the diagrammatic construction in \cite[Section 7.3]{KELP4} (see also \cite[Lemma 8.10]{KELP4}).
\end{defn}

We refer to \cite{KELP4} and \cite{ESW} for a diagrammatic exhibition of morphisms between Bott-Samelson bimodules. 
The morphism $\ELL([n,q])$ is determined by the condition that its diagram consists only of counterclockwise cups and right-facing crossings, as in the following diagram.

\[
{
\labellist
\small\hair 1pt
\pinlabel {$\leftred(q)$} [ ] at 60 50 
\pinlabel {$\rightred(q)$} [ ] at 245 50
\pinlabel {$I$} [ ] at 35 14 
\pinlabel {$M$} [ ] at 150 14
\pinlabel {$J$} [ ] at 280 14 
\endlabellist
\centering
\igb{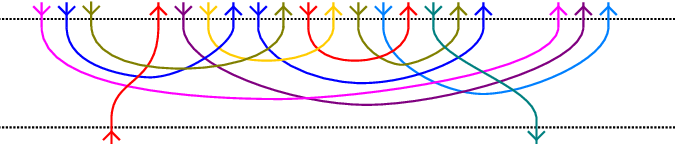}
} 
\]

In our examples, we color the simple reflection $\sberry$ in strawberry, and $\teal$ in teal. Sometimes $s = t$, which will force us to change our convention.

In type $A$ the expression $X_q^{\core}$ takes a simple form, and thus $\ELL([n,q])$ could be described more explicitly.

\begin{ex}\label{ex.ellA}
Let $W_M=\symm_{a+b}$ be a symmetric group, for some $a\neq b$, and let $I=\hat{\sberry}$ be such that $W_I=\symm_b\times \symm_a\subset \symm_{a+b}$. Then $J=\hat{\teal}$ is such that $W_J=\symm_a\times \symm_b\subset \symm_{a+b}$. As will be explained in  Section~\ref{s.A} (Equation \eqref{qk}), each $(I,J)$-coset $q$ in $W_M$ has a unique 
reduced expression $X_q$ of the form \eqref{eq.Xq}. 
\begin{enumerate}
\item If $q= W_I e W_J$ then we have
\[q\expr X_q = [[\hat{s}\supset \hat{s}\hat{t}]]\circ [\hat{s}\hat{t}]\circ [[\hat{s}\hat{t}\subset \hat{t}]]=[\hat{s}-\teal+\sberry].\]
Here we have 
\[\ELL([n,q])=\igs{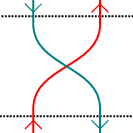}.\]
\item If $q = W_I w_M W_J$ then we have
\[q=\at \expr X_q =[I,M,J].\] Here we have $\ELL([n,q])=\id_{\BS([I,M,J])} $.
\item Otherwise, we have 
\[q \expr X_q = [[\hat{s} \supset \hat{s}{\hat{k}\hat{\ell}}]]\circ [\hat{s}{\hat{k}\hat{\ell}}\subset \hat{k}\hat{\ell}\supset \hat{t} \hat{k}\hat{\ell}]\circ[[\hat{t} \hat{k}\hat{\ell} \subset \hat{t}]] \]
for distinct $s,k,\ell,t\in M$. Here we have
\[ \ELL([n,q]) = \igs{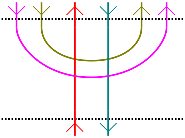}.\]
\end{enumerate}
\end{ex}

Here is a non-type A example. 

\begin{ex}
Let $(W,S)$ be of type $E_6$ where $S$ is indexed as in the Dynkin diagram 
\begin{tikzpicture}[scale=0.3,baseline=-3]
\protect\draw (8 cm,0) -- (6 cm,0);
\protect\draw (6 cm,0) -- (4 cm,0);
\protect\draw (4 cm,0) -- (2 cm,0);
\protect\draw (2 cm,0) -- (0 cm,0);
\protect\draw (4 cm,0) -- (4 cm,1.5 cm);

\protect\draw[fill=white] (8 cm, 0 cm) circle (.15cm) node[below=1pt]{\scriptsize $6$};
\protect\draw[fill=white] (6 cm, 0 cm) circle (.15cm) node[below=1pt]{\scriptsize $5$};
\protect\draw[fill=white] (4 cm, 0 cm) circle (.15cm) node[below=1pt]{\scriptsize $4$};
\protect\draw[fill=white] (2 cm, 0 cm) circle (.15cm) node[below=1pt]{\scriptsize $3$};
\protect\draw[fill=white] (4 cm, 1.5 cm) circle (.15cm) node[right=1pt]{\scriptsize $2$};
\protect\draw[fill=white] (0 cm, 0 cm) circle (.15cm) node[below=1pt]{\scriptsize $1$};
\end{tikzpicture}.
Let $M=S$ and $s=3$. Then $w_M3w_M=5$ and thus 
\[\at \expr [\hat{3}+3-5] = [\hat{3},M,\hat{5}]\] is an atom. For the $(\hat{3},\hat{5})$-coset $q<\at$ with reduced expression 
\[q\expr X_q= [[\hat{3}\supset \{4,6\}]]\circ [+3-4+5+2+1-6-2-3+4-5] \circ [[\{1,4\}\subset \hat{5}]], \]
the elementary light leaf $\ELL([n,q])$ is
\[ {
\labellist
\small\hair 2pt
\pinlabel {$1$} [ ] at 6 83
\pinlabel {$2$} [ ] at 18 83
\pinlabel {$5$} [ ] at 30 83
\pinlabel {$3$} [ ] at 46 83
\pinlabel {$4$} [ ] at 58 83
\pinlabel {$5$} [ ] at 70 83
\pinlabel {$2$} [ ] at 82 83
\pinlabel {$1$} [ ] at 94 83
\pinlabel {$6$} [ ] at 106 83
\pinlabel {$2$} [ ] at 118 83
\pinlabel {$3$} [ ] at 130 83
\pinlabel {$4$} [ ] at 142 83
\pinlabel {$5$} [ ] at 154 83
\pinlabel {$3$} [ ] at 170 83
\pinlabel {$2$} [ ] at 182 83
\pinlabel {$6$} [ ] at 194 83
\endlabellist
\centering
\ig{1}{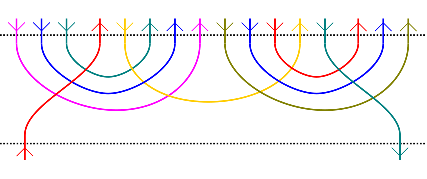}
} 
\]
\end{ex}





\subsubsection{Atomic double leaves}

There is a contravariant (but monoidally-covariant) ``duality'' functor $\calD$ from $\Frob$ to itself defined as follows: \begin{itemize}
\item it preserves objects and $1$-morphisms,
\item on $2$-morphisms, it flips each diagram upside-down and reverses all the orientations. \end{itemize}
This functor is an involution. 


\begin{defn}\label{def.DLL}
Given $q\leq \at$ and $b\in R^{\rightred(q)}$, the associated \emph{(right-sprinkled) double leaf} $\DLL_r(q,b)$ is the composition 
\begin{equation}\label{eq.dllq}
B_\at \xrightarrow{\ELL([n,q])}\BS(X_q)\xrightarrow{b} \BS(X_q)\xrightarrow{\calD(\ELL([n,q]))} B_\at.
\end{equation}
The middle map in \eqref{eq.dllq} uses that $X_q$ has the form \eqref{eq.Xq}: the map is multiplication by the element 
\[1^{\otimes}\otimes b\otimes 1\in \BS([[I\supset \leftred(q)]]\circ X_q^{\core})\otimes_{R^{\rightred(q)}} R^{\rightred(q)}\otimes_{R^J}R^J.\]
In diagrams, we have \[
\DLL_r(q,b)={
\labellist
\small\hair 1pt 
\pinlabel {$b$} [ ] at 245 64
\endlabellist
\centering
\igm{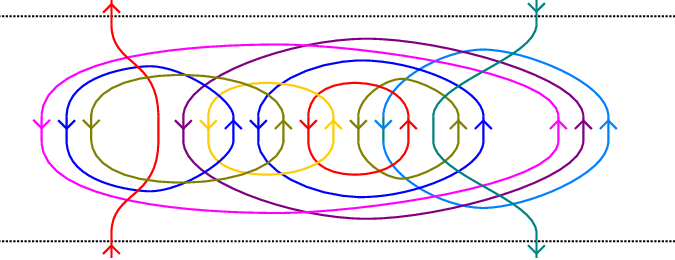}
} 
\]
\end{defn}

\begin{defn}
Given $q\leq \at$ and $b\in R^{\leftred(q)}$, the associated \emph{left-sprinkled double leaf} $\DLL_l(q,b)$ is the composition 
\begin{equation}\label{eq.dllql}
B_\at \xrightarrow{\ELL([n,q])}\BS(X_q)\xrightarrow{b} \BS(X_q)\xrightarrow{\calD(\ELL([n,q]))} B_\at,
\end{equation}
whose diagram is
\[
{
\labellist
\small\hair 1pt
\pinlabel {$b$} [ ] at 60 64 
\endlabellist
\centering
\igm{DLLat}
} .
\]
\end{defn}

The maps provided here (when $b$ ranges over a basis for $R^{\rightred(q)}$ or $R^{\leftred(q)}$) are the same as the ``double leaves basis'' from \cite{KELP4}. This is verified in the following remark, intended for a reader familiar with \cite{KELP4}.

\begin{rem}\label{rem.kelp4}
Let us verify  that
\begin{equation}\label{eq.DLLpaper4}
\DLL_l(q,b)=\DLL(q,([p,n,q],1),([p,n,q],1),b),
\end{equation} 
where $p$ is the identity $(I,I)$-coset, by following the stages in \cite[Chapter 7]{KELP4}.
First, we construct
$\LL(q,([p,n,q],1))$.
The single step light leaf for the first step $[p,n]$ is the identity map since $[p,n]$ is reduced. 
Since the left redundancy for $p$ and $n$ are the same, 
In the second step, we have the coset pair $[n,q]$ which is already a Grassmannian pair, thus the single step light leaf is the elementary light leaf  $\ELL([n,q])$. 

Then we choose $X_q=Y_q$ and take the dual map for the upside-down light leaf from $\BS(X_q)$ to $B_\at$. Altogether we obtain the double leaf of the form \eqref{eq.dllql}.  No rex moves were used at any stage of the process.

Moreover, for each $q\leq \at$, there is one subordinate path with terminus $q$, namely $[p,n,q]$. 
It follow that the morphisms $\DLL(q,([p,n,q],1),([p,n,q],1),b)$ form a double leaves basis in the sense of \cite{KELP4}. 
\end{rem}

\begin{rem}\label{rem:rightactiononDLL} We know that $\End(B_{\at})$ is an $(R^I, R^J)$-bimodule, so it is natural to ask how the actions of $R^I$ and $R^J$ interact with the bases presented in Propositions \ref{prop.DLLatom} and \ref{prop.DLLatomleft}. For $g \in R^J$ we claim that
\begin{equation} \label{rightactiononDLLr} \DLL_r(q,b) \cdot g = \DLL_r(q,b \cdot g). \end{equation}
Consider the diagram in Definition \ref{def.DLL}, and right-multiply by $g \in R^J$. Since $g$ is also in $R^{\rightred(q)}$, it can be slid from the right side to the region where $b$ lives.

However, the left action of $f \in R^I$ is more mysterious. For fixed $q$, it does not preserve the span of $\{\DLL_r(q,b) \mid b \in R^{\rightred(q)}\}$. Indeed, the comparison between the left action and the right action is controlled by polynomial forcing for $\BS(X_q)$, which involves lower terms.

Similarly, for $f \in R^I$, the left action on left-sprinkled double leaves is straightforward, 
\[ f \cdot \DLL_l(q,b) = \DLL_l(q,f \cdot b),\]
whereas the right action of $R^J$ is mysterious.
\end{rem}

\subsubsection{Evaluation of double leaves}\label{subsect:eval}

The following crucial computation links double leaves with the description $\End(B_\at) \cong R^I \ot_{R^M} R^J$.
Let $\Delta^J_{M,(1)}$ and $\Delta^J_{M,(2)}$ be dual bases of $R^J$ over $R^M$, where we use Sweedler notation.

\begin{lem}\label{lem.circleev}
The double leaf $\DLL_r(q,b)$  coincides with multiplication  by the element
\begin{equation}\label{eq.withb}
\pa_{\ma{q} w_J^{-1}}\left(b \cdot \Delta^J_{M,(1)} \right) \ot \Delta^J_{M,(2)}.
\end{equation}
This can also be written as
\begin{equation}
\pa_I^{\leftred(q)}\pa_{q^{\core}}(b \cdot \iota_J^{\rightred(q)} \Delta^J_{M,(1)})\ot \Delta^J_{M,(2)}. \end{equation}
\end{lem}

\begin{proof}
The proof is immediate from \cite[Algorithm 8.12]{KELP4}. It is proven exactly as \cite[Lemma 6.10]{KELP4}.
\end{proof}

A similar computation involving left-sprinkled double leaves gives the following.
\begin{lem}\label{lem.circleevleft}
The double leaf $\DLL_l(q,b)$  coincides with multiplication  by the element
\begin{equation}\pa_I^{\leftred(q)}(b \cdot \pa_{q^{\core}}(\iota_J^{\rightred(q)} \Delta^J_{M,(1)}))\ot \Delta^J_{M,(2)}.\end{equation}
\end{lem}

\begin{ex}
If $q\leq \at$ is the minimal $(I,J)$-coset, namely $q= W_IeW_J$, then we have two cases.
\begin{enumerate}
\item When $\textcolor{teal}{t}:=w_M \sberry w_M\neq \sberry$, we have \[q\expr [I,K,J]=[I - \textcolor{teal}{t}+\sberry]\] for $K=I\cap J$.
In this case, both left and right redundancies are $K$, and for $b\in R^K$ the double leaf $\DLL_r(q,b)=\DLL_l(q,b)$ is the left diagram in the equality
\begin{equation} \label{R2hard} 
	{
		\labellist
		\small\hair 2pt
		\pinlabel {$b$} [ ] at 31 32
		\endlabellist
		\centering
		\igm{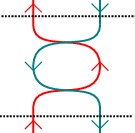}
	} =  \pa^K_I(b\cdot\Delta^J_{M,(1)})\igm{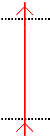} \igm{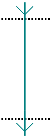}\Delta^J_{M,(2)} . \end{equation}
Thus we have
\[
\DLL_r(q,b)=\DLL_l(q,b)=
\pa^K_I(b\cdot\Delta^J_{M,(1)})\ot \Delta^J_{M,(2)}\]
(see \cite[Equations (89) and (91)]{KELP4}). 
\item When $w_M \sberry w_M=\sberry$, we have $q\expr [I]$. In this case, we have $\leftred(q)=I=\rightred(q)$, and for $b\in R^I$ the double leaf has a capcup diagram 
\begin{equation} \label{cupcap} {
		\labellist
		\small\hair 2pt
		\pinlabel {$b$} [ ] at 15 32
		\endlabellist
		\centering
		\igm{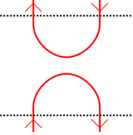}
	}.  \end{equation}
Thus we have $\DLL_r(q,b)=\DLL_l(q,b) = b\cdot \Delta_M^I= \Delta_M^I\cdot b$.
\end{enumerate}
\end{ex}

\begin{defn} \label{defn:PFfirstdef}
For an atomic coset $\at\expr [I,M,J]$, let $\PF_{<\at}$ denote the $\Bbbk$-linear subspace of $\End(B_{\at})$ spanned by right-sprinkled double leaves factoring through $q < \at$. More explicitly we have
\[\PF_{<\at}:= \Span_{q < \at} \{ \DLL_r(q,b) \mid b \in R^{\rightred(q)} \} = \Span_{q < \at} \{ \pa_{\ma{q}w_J^{-1}}
(R^{\rightred(q)}\cdot \Delta_{M,(1)}^J)\ot \Delta_{M,(2)}^J \}. \]
\end{defn}

\begin{lem} \label{lem:easycontain} We have $\PF_{<\at} \subset \End_{< \at}(B_{\at})$. \end{lem}

\begin{proof} By construction, every double leaf associated to $q < \at$ factors through a reduced expression for $q$, and thus lives in $\End_{< \at}(B_{\at})$. \end{proof}

We also note a consequence of Remark \ref{rem:rightactiononDLL}.

\begin{cor} \label{cor:rightaction} For each $g \in R^J$ and $b \in R^{\rightred(q)}$ we have
\begin{equation} \pa_{\ma{q}w_J^{-1}}
(b \cdot \Delta_{M,(1)}^J)\ot \Delta_{M,(2)}^J \cdot g = \pa_{\ma{q}w_J^{-1}}
(b \cdot g \cdot \Delta_{M,(1)}^J)\ot \Delta_{M,(2)}^J. \end{equation}
In particular, $\PF_{<\at}$ is a right $R^J$-module.
\end{cor}

\begin{proof} This follows from \eqref{rightactiononDLLr} and Lemma \ref{lem.circleev}. \end{proof}

\subsection{Double leaves and lower terms: part I}

The main result of \cite{KELP4} is that double leaves form a basis for morphisms between Bott-Samelson bimodules. However, \cite{KELP4} relies on Williamson's theory of standard filtrations, which relies on several assumptions originally made by Soergel.

\begin{defn} We call a realization a \emph{Soergel-Williamson realization} or an \emph{SW-realization} for short, if it is a Frobenius realization (see Definition \ref{defn:frobrealization}), and it also satisfies the following assumptions. \begin{itemize}
\item The realization is  \textit{reflection faithful}, i.e. it is faithful, and the reflections in
$W$ are exactly those elements that fix a codimension-one subspace.
\item The ring $\Bbbk$ is an infinite field of characteristic not equal to $2$. \end{itemize}
\end{defn}

\begin{rem}
Abe \cite{AbeSing} has recently developed a theory of singular Soergel bimodules that works for Frobenius realizations, without the extra restrictions of an SW-realization. One expects that the results of \cite{KELP4} can be straightforwardly generalized to Abe's setting. For simplicity and because of the current state of the literature, we will work with Williamson's  category of bimodules.
\end{rem}

\begin{prop}\label{prop.DLLatom}
Assume an SW-realization. Let $\mathbb B_q$ be a $\Bbbk$-basis of $R^{\rightred(q)}$, for each $q\leq \at$.
Then 
\begin{equation}\label{eq.dllb}
\{\DLL_r(q,b)\}_{b\in \mathbb B_q}
\end{equation}
gives a basis of $\End_{\leq q}(B_{\at})/\End_{<q}(B_{\at})$ over $\Bbbk$.
In particular, 
\begin{equation}\label{eq.dllat}
\{\DLL_r(q,b)\ |\ q\le \at, b\in \mathbb B_q\}
\end{equation}
is a $\Bbbk$-basis of $\End(B_{\at})$, and the subset indexed by $q < \at$ is a basis for $\End_{< \at}(B_{\at})$.
\end{prop}

\begin{proof}
This is a special case of  \cite[Lemma 8.35 and Theorem 7.49]{KELP4}, as confirmed in Remark~\ref{rem.kelp4}. \end{proof}


\begin{cor} \label{cor:pfisltoverSW} Assume an SW-realization. We have 
\begin{equation}  \label{lowertermsummary} \PF_{<\at} = \End_{<\at}(B_{\at}) =
\bigoplus_{q<\at} \pa_{\ma{q}w_J^{-1}}
(R^{\rightred(q)}\cdot \Delta_{M,(1)}^J)\ot \Delta_{M,(2)}^J.\end{equation}
In particular, $\PF_{<\at}$ is an $(R^I, R^J)$-bimodule.
\end{cor}

\begin{proof} As $\PF_{<\at}$ is the span of double leaves factoring through $q < \at$, the first equality follows from Proposition \ref{prop.DLLatom}. The second equality follows from the linear independence of double leaves. Since $\End_{<\at}(B_{\at})$ is an $(R^I, R^J)$-bimodule, so is $\PF_{<\at}$. \end{proof}

Similarly, double leaves provides a left-sprinkled basis.

\begin{prop}\label{prop.DLLatomleft}
Assume an SW-realization. Let $\mathbb B_q$ be a $\Bbbk$-basis of $R^{\leftred(q)}$, for each $q\leq \at$.
Then 
\begin{equation}\label{eq.dllbl}
\{\DLL_l(q,b)\}_{b\in \mathbb B_q}
\end{equation}
gives a basis of $\End_{\leq q}(B_{\at})/\End_{<q}(B_{\at})$ over $\Bbbk$.
In particular, 
\begin{equation}\label{eq.dllatl}
\{\DLL_l(q,b)\ |\ q\leq \at, b\in \mathbb B_q\}
\end{equation}
is a $\Bbbk$-basis of $\End(B_{\at})$, and the subset indexed by $q < \at$ is a basis for $\End_{< \at}(B_{\at})$.
\end{prop}

\subsection{Double leaves and lower terms: part II} \label{ssec:localizationtricks}

\begin{defn} \label{defn:almostSW} An \emph{almost-SW realization} is a Frobenius realization, together with the following assumptions. \begin{itemize}
\item The ring $\Bbbk$ is a domain with fraction field $\ffk$.
\item After base change to $\ffk$, the result is an SW-realization.
\item  Finitely-generated projective modules over $\Bbbk$ are free.
\end{itemize}
\end{defn}

\begin{ex}  The defining representation of $\symm_n$ over $\Z$ is an almost-SW realization \cite[Lemma 5]{Demazure}. \end{ex}

\begin{ex} The root realization of a Weyl group is almost-SW when defined over $\Bbbk= \Z[1/N]$ for small $N$ ($N=30$ will suffice for all Weyl groups by \cite[Proposition 8]{Demazure}). \end{ex}

\begin{lem}\label{lem:RIisfree}  Let $I \subset S$ be finitary. Then $R^I$ is a free $\Bbbk$-module. \end{lem}

\begin{proof} As a polynomial ring over a free $\Bbbk$-module $V$, $R$ is a free $\Bbbk$-module. By the assumption of generalized Demazure surjectivity, $R$ is free as an $R^I$-module when $I \subset S$ is finitary. Thus $R^I$ is a direct summand of $R$, and is therefore projective as a $\Bbbk$-module. Since both $R$ and $R^I$ are finitely-generated as $\Bbbk$-modules in each graded degree, we deduce that $R^I$ is also a free module over $\Bbbk$. \end{proof}


Our goal in this section is to generalize the results of the previous section to almost-SW realizations. In all the lemmas in this section, we assume an almost-SW realization. First we note the compatibility of base change with most of the constructions above.

We let $R$ be the polynomial ring of the realization over $\Bbbk$, and let $R_\ffk :=  R\otimes_{\Bbbk}\ffk$ be the polynomial ring of the realization after base change. Let $R^I_\ffk \subset R_{\ffk}$ be the invariant subring.

\begin{lem} We have $R^I_{\ffk} \cong R^I \ot_{\Bbbk} \ffk$. \end{lem}

\begin{proof} 


There is a natural map $R^I \ot_{\Bbbk} \ffk \to R_{\ffk}$, and since scalars are $W$-invariant, the image lies within $R_{\ffk}^I$. The map is injective since $\ffk$ is flat over $\Bbbk$. We now argue that the map $R^I \ot_{\Bbbk} \ffk \to R_{\ffk}^I$ is surjective. If $f \in R_{\ffk}^I$, then there is some $c \in \Bbbk$ such that $cf \in R$ (e.g. letting $c$ be the product of the denominators of each monomial in $f$). Clearly $cf \in R^I$, whence $f$ is the image of $cf \ot \frac{1}{c}$.

\end{proof}

Let $B_{\at,\ffk}=R^I_\ffk\otimes_{R^M_\ffk}R^J_\ffk$. If
$I_{\bullet} = [[I= I_0 \subset K_1 \supset I_1 \subset \ldots K_m \supset I_m = J]]$ is a multistep $(I,J)$-expression, let 
\[ \BS_\ffk(I_{\bullet}) := R^{I_0}_\ffk \ot_{R_\ffk^{K_1}} R_\ffk^{I_1} \ot_{R_\ffk^{K_2}} \cdots \ot_{R_\ffk^{K_m}} R_\ffk^{I_m}.\]

\begin{lem}\label{lem.BSloc}
The natural inclusion map $\BS(I_{\bullet}) \to \BS(I_{\bullet}) \ot_{\Bbbk} \ffk$ is injective.
We have $\BS_\ffk(I_\bullet)\cong \BS(I_\bullet) \otimes_{\Bbbk}\ffk$. As a consequence we have an injective map
\begin{equation} \label{eq:basechangehom} \Hom(\BS(I_{\bullet}),\BS(I'_{\bullet})) \to \Hom(\BS_{\ffk}(I_{\bullet}),\BS_{\ffk}(I'_{\bullet})). \end{equation}
\end{lem}

\begin{proof}


By our assumptions from \S\ref{ss.realizations}, $R^I$ is free over $R^K$ whenever $I \subset K$. We fix a basis $\{b_i^{I,K}\}$ for this extension.
Hence the Bott-Samelson bimodule $\BS(I_\bullet)$ is free as a right $R^{I_m}$-module with basis 
\begin{equation}\label{BSbasis}
    \left\{b_{i_1}^{I_0,K_1}\ot \ldots \ot b_{i_m}^{I_{m-1},K_m} \ot 1\right\}.
\end{equation}
It is also free as a right $\Bbbk$-module by \Cref{lem:RIisfree}. So base change is injective on bimodules.
Notice that $\{b_i^{I,K}\ot 1\}$ is a basis of $R^I_\ffk$ over $R^K_\ffk$. Hence 
\eqref{BSbasis}
gives also a basis of $\BS_{\ffk}(I_\bullet)$ over $R^{I_m}_\ffk$.
Since it sends a basis to a basis, we deduce that the natural map $\BS(I_\bullet) \otimes_{\Bbbk}\ffk\to\BS_\ffk(I_\bullet)$ is an isomorphism.

The localization functor gives the map in \eqref{eq:basechangehom}. For a morphism $\phi$ between bimodules over $\Bbbk$, let $\phi \ot 1$ denote its image, a morphism between bimodules over $\ffk$. The restriction of $\phi \ot 1$ to the subset $\BS(I_{\bullet}) \subset \BS_{\ffk}(I_{\bullet})$ is the original morphism $\phi$. Hence if $\phi \ot 1$ is the zero morphism, so is $\phi$.
\end{proof}



\begin{lem} For $q \le \at$ and $b \in R^{\rightred(q)}$, let us temporarily write $\DLL_{r,\Bbbk}(q,b)$ for the double leaf as a morphism between Bott-Samelson bimodules over $\Bbbk$, and $\DLL_{r,\ffk}(q,b)$ for the double leaf as a morphism between Bott-Samelson bimodules over $\ffk$, where for the latter we identify $b$ with its image in $R_{\ffk}^{\rightred(q)}$. Under the map \eqref{eq:basechangehom}, we have 
\[ \DLL_{r,\Bbbk}(q,b) \mapsto \DLL_{r,\ffk}(q,b).\]
\end{lem}

\begin{proof} The calculus from \cite{ESW} for interpreting diagrams is invariant under base change. Alternatively, $B_{\at}$ is generated as a bimodule by $1 \ot 1$, and elementary light leaves $\ELL([n,q])$ are determined uniquely in their Hom space by the fact that they send $1\otimes 1$ to $1\ot \ldots \ot 1$. This property is preserved by base change. \end{proof}

\begin{lem} Let $\mathbb B_q$ be a basis of $R^{\rightred(q)}$ over $\Bbbk$. Then the set $\{\DLL_{r,\Bbbk}(q,b)\ |\ q\le \at, b\in \mathbb B_q\}$ is linearly independent. Let $\PF_{<\at,\Bbbk}$ and $\PF_{<\at,\ffk}$ be defined as before for their respective realizations. The map \eqref{eq:basechangehom} induces an isomorphism
\begin{equation} \label{eq:pfisom} \PF_{<\at,\Bbbk} \ot_{\Bbbk}\, \ffk \simto \PF_{<\at,\ffk}. \end{equation}
\end{lem}

\begin{proof} By definition $\PF_{< \at}$ is the span (over $\Bbbk$ or $\ffk$) of the double leaf morphisms. The map \eqref{eq:basechangehom} restricts to a map $\PF_{< \at, \Bbbk} \to \Hom(\BS_{\ffk}(I_{\bullet}),\BS_{\ffk}(I'_{\bullet}))$. By the previous lemma, the image of this map is contained in $\PF_{< \at, \ffk}$. Thus one has an induced map $\PF_{<\at, \Bbbk} \ot_{\Bbbk} \ffk \to \PF_{<\at,\ffk}$.

Note that $\mathbb B_q$ is sent by base change to a basis of $R^{\rightred(q)}_{\ffk}$ over $\ffk$. The elements $\{\DLL_{r,\Bbbk}(q,b) \ot 1 \}$ (ranging over the appropriate index set) form an $\ffk$-spanning set for the left-hand side of \eqref{eq:pfisom}, and are sent to $\{\DLL_{r,\ffk}(q,b)\}$, which form a $\ffk$-basis for $\PF_{< \at, \ffk}$ by Proposition \ref{prop.DLLatom}. Thus the map \eqref{eq:pfisom} is an isomorphism, and the elements $\{\DLL_{r,\Bbbk}(q,b) \ot 1\}$ are linearly independent over $\ffk$. Consequently, $\{\DLL_{r,\Bbbk}(q,b)\}$ are linearly independent over $\Bbbk$. \end{proof}




\begin{lem} \label{lem:linindep} We have
\[ \PF_{<\at,\Bbbk} = \bigoplus_{q<\at} \pa_{\ma{q}w_J^{-1}}
(R^{\rightred(q)}\cdot \Delta_{M,(1)}^J)\ot \Delta_{M,(2)}^J.\]
\end{lem}

\begin{proof} The span in Definition \ref{defn:PFfirstdef} in indeed a direct sum of subspaces, by the linear independence shown in the previous lemma. \end{proof}

Henceforth we return to the $\Bbbk$-linear setting by default (i.e. in the absence of a subscript). Now the question remains: is the inclusion $\PF_{<\at} \subset \End_{< \at}(B_{\at})$ an equality over $\Bbbk$, knowing that the result holds over $\ffk$? In Corollary \ref{cor.ST=LT} below, we prove that the answer is yes.

\begin{lem}\label{lem:othereasything} For any $\phi \in \End_{< \at}(B_{\at})$ there is some $n \in \Bbbk$ such that $n \phi \in \PF_{<\at}$. \end{lem}

\begin{proof} In view of \Cref{lem.BSloc} and \eqref{eq:basechangehom}, we can regard both $\PF_{<\at}$ and $\End(B_{\at})$ as $\Bbbk$-submodules of $\End(B_{\at,\ffk})$, which we identify with $B_{\at,\ffk}$. We have $\End_{< \at}(B_{\at}) \subset \End_{< \at}(B_{\at,\ffk})$ by definition. Meanwhile, \Cref{cor:pfisltoverSW} holds over $\ffk$ and thus 
\begin{equation}  \label{lowertermsummary2} \End_{<\at}(B_{\at,\ffk}) 
\cong \bigoplus_{q<\at} \pa_{\ma{q}w_J^{-1}}
(R_\ffk^{\rightred(q)}\cdot \Delta_{M,(1)}^J)\ot \Delta_{M,(2)}^J. \end{equation}
So any $\phi \in \End_{<\at}(B_\at)$ is a $\ffk$-multiple of an element of $\PF_{<\at}$. Multiplying by the denominator, there exists $n \in \Bbbk$ such that $n\phi \in \PF_{<\at}$. \end{proof}

We continue with a divisibility lemma on Demazure operators, which ensures that Demazure operators do not ``produce'' additional divisibility by elements of $\Bbbk$.

\begin{lem}\label{Demazuredivisibility}
Let $b\in R$ and let $n\in \Bbbk$. If $n\mid \partial_{w_Mw_J^{-1}}(bg)$ for all $g\in R^J$, then $n\mid b$.
\end{lem}

\begin{proof}
Assume that $n\nmid b$. We need to find $g\in R^J$ such that $n\nmid\pa_{w_Mw_J^{-1}}(bg)$.

Recall from Lemma \ref{lem:iteratedleibnizforw} that for $w \in W$ and $f, g \in R$ we have
\[ \pa_w(fg) = \sum_{x \leq w} T'_x(f) \pa_x(g)\]
for certain operators $T'_x$ defined over $\Bbbk$, where $T'_w(f) = w(f)$. 
Let $W^J$ be the subset of elements in $W$ which are minimal (for the Bruhat order) in their right $W_J$-coset. We have $\partial_x(g)=0$ when $x\not \in W^J$ and $g \in R^J$, so we have 
\[\pa_w(fg) = \sum_{x \leq w, x\in W^J} T'_x(f) \pa_x(g).\]

We apply this formula when $w = w_M w_J^{-1}$. Let $W^J_M = W^J \cap W_M$. Since $n\nmid b$, then also $n\nmid w_Mw_J^{-1}(b)=T_{w_Mw_J^{-1}}'(b)$. So there exists some $y\in W_M^J$ (not necessarily unique) which is minimal with respect to the property that $n\nmid T'_y(b)$. Now we have
\begin{equation}\label{Leibnizmodn}
\pa_{w_Mw_J^{-1}}(bg)= \sum_{x \in W_M^J} T'_x(b)\pa_x(g)\equiv \sum_{x\not< y, x \in W_M^J} T'_x(b)\pa_x(g) \pmod n. 
\end{equation}

Recall that $\ell(xw_M)=\ell(w_M)-\ell(x)$ for all $x\in W_M$.
Let $z=y^{-1}w_M$ so that $y.z=w_M$. We have
\[ \ell(w_Jy^{-1}w_M)=\ell(w_M)-\ell(w_Jy^{-1})=\ell(w_M)-\ell(y^{-1})-\ell(w_J)=\ell(z)-\ell(w_J),\]
thus the left descent set of $z$ contains $J$. By \cite[Lemma 3.2]{KELP1} we have  $\im(\pa_z)\subset R^J$. By our assumption of generalized Demazure surjectivity, we can choose $P_M\in R$ such that $\pa_{w_M}(P_M) = 1$. Set $g=\pa_z(P_M)\in R^J$ and note that $\pa_y(g)=1$.

Let $x\in W_M$. If $x.z$ is not reduced, then $\pa_x(g)= \pa_x\pa_z(P_M) = 0$  by \cite[(27))]{KELP1}. If $x.z$ is reduced, then 
\[\ell(w_M)-\ell(yx^{-1})=\ell(x.y^{-1}w_M)=\ell(x)+\ell(w_M)-\ell(y).\]
It follows that $\ell(yx^{-1})+\ell(x)=\ell(y)$, so $yx^{-1}.x=y$ and, in particular, $x\leq y$.
This means that $\pa_x(g)\neq 0$ only if $x\leq  y$.

Finally, we plug $g=\pa_z(P_M)$  into \eqref{Leibnizmodn} and we observe that $\pa_{w_Mw_J^{-1}}(bg)\equiv T'_y(b) \not \equiv 0 \pmod n$.
\end{proof}

Now we can prove that $\PF_{<\at}$, as a submodule of $\End(B_\at)$, is closed under division by elements in $\Bbbk$ (when that makes sense).

\begin{prop}\label{prop:smallertermsoverZ}
Let $\phi \in \PF_{<\at}$ and assume there exists  $n\in \Bbbk$ such that $\frac{1}{n}\phi \in \End(B_\at)$. Then $\frac{1}{n}\phi \in \PF_{<\at}$.
\end{prop}

\begin{proof}
If $n$ is a unit in $\Bbbk$ the result is trivial, so assume otherwise.

Let $\phi\in \PF_{<\at}$ and assume that $\frac{1}{n}\phi\in \End(B_\at)\cong R^I\otimes_{R^M}R^J$. We can write
\[\phi = \sum_{q<\at} \pa_{\ma{q}w_J^{-1}}
	(b_q\cdot \Delta_{M,(1)}^J)\ot \Delta_{M,(2)}^J\]
for some unique $b_q\in R^{\rightred(q)}$. For clarity we choose to unravel Sweedler's notation. Choose dual bases $\{c_i\}$ and $\{d_i\}$ for $R^J$ over $R^M$ relative to the Frobenius trace map $\pa^J_M$. We have
\[\phi = \sum_i\sum_{q<\at} \pa_{\ma{q}w_J^{-1}}
	(b_q\cdot c_i)\ot d_i.\]

Since $d_i$ is a basis of $R^J$ over $R^M$, any element of $R^I \ot_{R^M} R^J$ is uniquely expressible as $\sum_i f_i \ot d_i$ for $f_i \in R^I$. In particular, if $n$ divides $\sum_i f_i \ot d_i$ we have 
\[\frac{1}{n} \sum_i f_i \ot d_i=\sum_i f'_i \ot d_i \]
for some unique $f'_i$. Then $\sum f_i \ot d_i =\sum_i nf'_i \ot d_i$, and by the unicity mentioned before, this implies that $nf_i'=f_i$ for all $i.$

Consequently, $\phi$ is divisible by $n$ if and only if for all $i$ we have
\[n \mid\sum_{q<\at} \pa_{\ma{q}w_J^{-1}}
	(b_q\cdot c_i).\]
We want to show that all the $b_q$ are actually divisible by $n$, so that $\frac{1}{n}\phi \in \PF_{<\at}$.
	
Assume for contradiction that there exists a minimal $r$ such that $n\nmid b_r$. Let $z$ be such that $z.\ma{r}w_J^{-1}=w_Mw_J^{-1}$. Then by similar arguments to the previous lemma we have 
\[0\equiv \pa_z\left(\sum_{q<\at} \pa_{\ma{q}w_J^{-1}}
	(b_q\cdot c_i)\right)\equiv\pa_{w_Mw_J^{-1}}(b_r \cdot c_i)\pmod n\]
 for each $i$.
The subset of $R^J$ consisting of those $c$ for which $\pa_{w_M w_J^{-1}}(b_r \cdot c) \equiv 0 \pmod n$ is evidently an $R^M$-submodule. Since this submodule contains a basis $\{c_i\}$ for $R^J$ over $R^M$, it must contain all of $R^J$. Thus
\[ n\mid \pa_{w_Mw_J^{-1}}(b_r \cdot g)\]
for all $g \in R^J$. By \Cref{Demazuredivisibility} we deduce $n\mid b_r$, leading to a contradiction.
\end{proof}

\begin{cor}\label{cor.ST=LT}
For an almost SW-realization we have $\PF_{<\at}=\End_{<\at}(B_\at)$. In particular, $\PF_{<\at}$ is an $(R^I, R^J)$-bimodule.
\end{cor}

\begin{proof}
The containment $\PF_{<\at} \subset \End_{\at}(B_{\at})$ was already shown in \Cref{lem:easycontain}. Now pick an arbitrary element $\psi\in \End_{<\at}(B_\at)$. By \Cref{lem:othereasything}, there is some $n \in \Bbbk$ such that $n\psi \in \PF_{<\at}$. Then $\psi=\frac{1}{n}(n\psi)\in \End_{<\at}(B_\at)$, so by \Cref{prop:smallertermsoverZ} we deduce that $\psi\in \PF_{<\at}$.
\end{proof}





\section{Polynomial forcing and atomic Leibniz}\label{s.PF}


\subsection{Polynomial forcing for atomic cosets} \label{subsec:polyforceatomic}

Now we explain the concept of polynomial forcing. We consider first the case of an atomic coset $\at\expr [[I\subset M\supset J]]$.

Recall that $\mi{\at}J = I\mi{\at} $ and $ \leftred(\at)=I$ since $\at$ is core. Thus there is an isomorphism
\begin{equation} R^J \to R^I, \qquad f \mapsto \mi{\at}(f). \end{equation}

\begin{defn} \label{defn:polyforceforp} Let $\at$ be an atomic coset with reduced expression $[[I \subset M \supset J]]$ and let $f \in R^J$. We say that \emph{polynomial forcing} holds for $f$ and $\at$ if we have
\begin{equation} \label{polyforceatomic} 1 \ot f - \mi{\at}(f) \ot 1 \in \End_{< \at}(B_\at). \end{equation}
We say that \emph{polynomial forcing} holds for $\at$ if \eqref{polyforceatomic} holds for all $f \in R^J$.
\end{defn}

\begin{lem} \label{lem:polyforceaddmult} Suppose that \eqref{polyforceatomic} holds for $f_1$ and for $f_2$, with both $f_1, f_2 \in R^J$. Then it holds for $f_1 + f_2$ and $f_1 f_2$. \end{lem}

\begin{proof} Additivity is trivial, because $\Hom_{< \at}$ is closed under addition. Now consider the following:
\begin{equation} 1  \ot f_1 \cdot f_2 - \mi{\at}(f_1 \cdot f_2) \ot 1 = (1 \ot f_1 - \mi{\at}(f_1) \ot 1) \cdot f_2 + \mi{\at}(f_1) \cdot (1 \ot f_2 - \mi{\at}(f_2) \ot 1). \end{equation}
Since $\End_{<\at}(B_\at)$ is closed under right and left multiplication, both terms on the right-hand side above are in $\End_{<\at}(B_\at)$, and the result is proven. \end{proof}

Before continuing, let us contrast polynomial forcing with an a priori different notion.

\begin{defn} Consider $\PF_{<\at}$ from Definition \ref{defn:PFfirstdef}. We say that \emph{$\PF$-forcing} holds  for $\at$ and $f$ if $1 \ot f - \mi{\at}(f) \ot 1 \in \PF_{<\at} \subset \End(B_{\at})$. We say that \emph{$\PF$-forcing} holds for $\at$ if it holds for $\at$ and $f$, for all $f\in R^J.$ \end{defn}

For an almost-SW realization, $\PF$-forcing is equivalent to polynomial forcing, since $\PF_{<\at} = \End_{< \at}(B_{\at})$ by Corollary \ref{cor.ST=LT}. In general, it is not obvious that $\PF$-forcing is multiplicative. The proof of multiplicativity in Lemma \ref{lem:polyforceaddmult} relied on the fact that $\End_{<\at}(B_{\at})$ is an $(R^I, R^J)$-bimodule, whereas $\PF_{<\at}$ is only a priori a right $R^J$-module.




\subsection{Equivalence} \label{ssec:equivalence}

Now we prove the equivalence between atomic Leibniz rules and polynomial forcing. To formulate an intermediate condition in the proof, which is also  of importance for the next section, we agree to say the following.
Given an atomic $(I,J)$-coset $\at$ and an element $f\in R^J$, \emph{an atomic Leibniz rule for $\at$ and $f$} is said to hold if there exist elements $T_q(f)$ such that equation~\eqref{eq:atomicleibniz}  is satisfied for all $g\in R^J$. Since this condition is stated for one polynomial $f$ at a time, there is no requirement that $T_q$ is an $R^M$-linear operator.

\begin{prop}\label{prop:AL=PF}
	Let $\at\expr [I,M,J]$ be an atomic $(I,J)$-coset, and $f \in R^J$.
	We have a rightward atomic Leibniz rule for $\at$ and $f$ if and only if $\PF$-forcing holds for $\at$ and $f$. Moreover, for an almost-SW realization, 
 if $\PF$-forcing holds for $\at$ and $f$, then 
the atomic Leibniz rule 
is unique, i.e., the elements $T_q(f)\in R^{\rightred(q)}$ in \eqref{eq:atomicleibniz} are uniquely determined.
\end{prop}

\begin{proof}
Note that $R^J\subset R^M$ is a Frobenius extension, see \cite[Section 24.3.2]{GBM}. The trace map is
\[\pa^J_M := \pa_{w_M w_J^{-1}} = \pa_\at.\] 
Let $\Delta^J_M \in R^J \ot_{R^M} R^J$ denote the coproduct element (the image of $1 \in R^J$ under the coproduct map), which we often denote using Sweedler notation. Then \cite[(2.2) with $f=1$]{ESW} implies that
\begin{equation} \label{eqforonetensor} 1 \ot 1 = \pa_\at(\Delta^J_{M,(1)}) \ot \Delta^J_{M,(2)}. \end{equation}
Multiplying $\mi{\at}(f)$ on the left 
we get
\begin{equation*} \mi{\at}(f) \ot 1 = \mi{\at}(f) \cdot \pa_\at(\Delta^J_{M,(1)}) \ot \Delta^J_{M,(2)}. \end{equation*}
Meanwhile, \cite[(2.2)]{ESW} implies that
\begin{equation} \label{eqforfotone} 1 \ot f = \pa_\at(f \cdot \Delta^J_{M,(1)}) \ot \Delta^J_{M,(2)}. \end{equation}
Thus we have
\begin{equation} \label{eq:differenceis} 1 \ot f - \mi{\at}(f) \ot 1 = \left[ \pa_\at(f \cdot \Delta^J_{M,(1)}) - \mi{\at}(f) \cdot \pa_\at(\Delta^J_{M,(1)}) \right] \ot \Delta^J_{M,(2)}. \end{equation}

Letting $g = \Delta^J_{M,(1)}$, a rightward atomic Leibniz rule for $\at$ and $f$ gives
\begin{equation} \pa_{\at}(f \cdot g) - \mi{\at}(f) \cdot \pa_{\at}(g) = \sum_{q<\at}\pa_{\ma{q}w_J^{-1}}(T_q(f) \cdot g). \end{equation}
Thus we have
\begin{equation} \label{eq:differenceisagain} 1 \ot f - \mi{\at}(f) \ot 1 = \sum_{q<\at}\pa_{\ma{q}w_J^{-1}}\left(T_q(f) \cdot \Delta^J_{M,(1)} \right) \ot \Delta^J_{M,(2)}. \end{equation}
which lies in $\PF_{<\at}$  by definition.

We prove now the other direction. We have
\begin{equation*} 1 \ot f - \mi{\at}(f) \ot 1 \in \PF_{<\at}. \end{equation*}
By \eqref{eq:differenceis}, we obtain
\begin{equation*} \left[ \pa_\at(f \cdot \Delta^J_{M,(1)}) - \mi{\at}(f) \cdot \pa_\at(\Delta^J_{M,(1)}) \right] \ot \Delta^J_{M,(2)} \in \End_{< \at}(B_\at). \end{equation*}
By definition of $\PF_{<\at}$ we deduce that
\begin{align} \nonumber \pa_\at(f \cdot \Delta^J_{M,(1)}) & \ot \Delta^J_{M,(2)} = \\
	\label{atomicbuttensor'}& \mi{\at}(f) \cdot \pa_\at(\Delta^J_{M,(1)}) \ot \Delta^J_{M,(2)} + \sum_{q<\at} \pa_{\ma{q}w_J^{-1}}(T_q(f) \cdot \Delta_{M,(1)}^J) \ot \Delta_{M,(2)}^J \end{align}
for some $T_q(f) \in R^{\rightred(q)}$. For an almost-SW realization, Lemma \ref{lem:linindep} implies that the  $T_q(f)$ are unique.

Note that $\Delta^J_{M,(1)}$ and $\Delta^J_{M,(2)}$ run over dual bases of $R^J$ over $R^M$. The elements $\mathbb{B} = \{1 \ot \Delta^J_{M,(2)}\}$ form a basis for $B_\at$, when viewed as a left $R^I$-module. Thus in order for the equation \eqref{atomicbuttensor'} to hold, it must be an equality for each coefficient with respect to the basis $\mathbb{B}$. Hence we conclude
\begin{equation} \label{atomicbutdelta'} \pa_\at(f \cdot \Delta) = \mi{\at}(f) \cdot \pa_\at(\Delta) + \sum_{q<\at} \pa_{q}(T_q(f) \cdot \Delta) \end{equation}
for all $\Delta$ ranging through a basis of $R^J$ over $R^M$.

Using the linearity of \eqref{atomicbutdelta'} over $R^M$, we deduce that it continues to hold when $\Delta$ is replaced by any element $g \in R^J$. Thus the atomic Leibniz rule for $f$ is proven.
\end{proof}

\begin{thm}\label{thm.AL=PF}
Assume an almost-SW realization (see Definition \ref{defn:almostSW}). Let $\at\expr[I,M,J]$ be an atomic $(I,J)$-coset. Then the following are equivalent.
\begin{enumerate}
\item\label{AL} A rightward atomic Leibniz rule holds for $\at$.
\item\label{lAL} A leftward atomic Leibniz rule holds for $\at$.
\item\label{ALc} For a set of generators $\{c_i\}$ of the $R^M$-algebra $R^J$, a rightward atomic Leibniz rule holds for $\at$ and each $c_i$.
\item\label{lALc} For a set of generators $\{c_i\}$ of the $R^M$-algebra $R^J$, a leftward atomic Leibniz rule holds for $\at$ and each $c_i$.
\item\label{PF} Polynomial forcing holds for $\at$.
\end{enumerate}
Moreover, in this case, there are unique operators $T_q, T'_q$ that satisfy atomic Leibniz rules.
\end{thm}

\begin{proof}
First, we observe that polynomial forcing holds for $f \in R^M$. Clearly $1 \ot f = f \ot 1$. Moreover, $\at \subset W_M$ so $\mi{\at}(f) = f$.

That \eqref{AL} implies \eqref{ALc} is clear.

Suppose that \eqref{ALc} holds. By Proposition \ref{prop:AL=PF}, $\PF$ forcing holds for all $c_i$. By Corollary \ref{cor.ST=LT}, $\PF$-forcing for $c_i$ is equivalent to polynomial forcing for $c_i$. By Lemma \ref{lem:polyforceaddmult}, the subset of $R^J$ consisting of those $f$ for which polynomial forcing holds is a subring. As explained above, this subring includes $R^M$, so if it includes $\{c_i\}$ then it must be all of $R^J$. In this way, \eqref{ALc} implies \eqref{PF}.

Suppose that \eqref{PF} holds. Once again, Proposition \ref{prop:AL=PF} and Corollary \ref{cor.ST=LT} imply that, for each $f \in R^J$, a rightward atomic Leibniz rule holds for $f$, with the elements $T_q(f) \in R^{\rightred(q)}$ being unique. To prove that a rightward atomic Leibniz rule holds, it remains to prove that the operators $T_q \co R^J \to R^{\rightred(q)}$ are $R^M$-linear. We do this below, finishing the proof that \eqref{PF} implies \eqref{AL}.

Let $g \in R^M$. Multiplying both sides of equation \eqref{atomicbuttensor'} on the left by $g$, and pulling $g$ into various $R^M$-linear operators (namely $\pa_{\at}$ and $\mi{\at}$ and $\pa_{\ma{q} w_J^{-1}}$), we obtain
\begin{align} \nonumber \pa_\at(gf \cdot \Delta^J_{M,(1)}) & \ot \Delta^J_{M,(2)} = \\
	&\nonumber \mi{\at}(gf) \cdot \pa_\at(\Delta^J_{M,(1)}) \ot \Delta^J_{M,(2)} + \sum_{q<\at} \pa_{\ma{q}w_J^{-1}}(gT_q(f) \cdot \Delta_{M,(1)}^J) \ot \Delta_{M,(2)}^J \end{align}
This is exactly \eqref{atomicbuttensor'} with $gf$ replacing $f$, except that $gT_q(f)$ appears instead of $T_q(gf)$. 
By uniqueness, we deduce that $T_q(gf) = g T_q(f)$. 

We have thus shown the equivalence of \eqref{AL},\eqref{ALc} and \eqref{PF}.
A similar argument will imply the equivalence of \eqref{lAL} and \eqref{lALc} and \eqref{PF}, and the uniqueness of $T'_q$. This similar argument replaces $\DLL_r(q,b)$ with $\DLL_l(q,b)$, using Proposition~\ref{prop.DLLatomleft} and  Lemma~\ref{lem.circleevleft}. The left analogue of the remaining arguments (e.g. Proposition \ref{prop:AL=PF} and Corollary \ref{cor.ST=LT}) is left to the reader.
\end{proof}

\begin{rem} The intermediate conditions \eqref{ALc} and \eqref{lALc} do not play a significant role in the proof. We have included them to make it easier to prove the atomic Leibniz rule by establishing it on a set of generators. \end{rem}

\subsection{Polynomial forcing for general cosets} \label{subsec:polyforcegeneral}


Now let $p$ be an arbitrary $(I,J)$-coset, with a reduced expression $I_{\bullet}$.
We wish to avoid the technicalities of changing the reduced expression $I_{\bullet}$ in this paper. Instead we focus on the special case when $I_{\bullet}$ is an \emph{atomic-factored reduced expression}, i.e. it has the following form:
\begin{equation} \label{specialexpression} I_{\bullet} = [[I \supset \leftred(p)]] \circ I'_{\bullet} \circ [[\rightred(p) \subset J]] \end{equation}
where $I'_{\bullet}$ is an \emph{atomic reduced expression} (see below) for $p^{\core}$. 

An atomic reduced expression for a core coset $p^{\core}$ is a reduced expression of the form 
\begin{equation} \label{atomicrex} I'_{\bullet} = [[\leftred(p) = N_0 \subset M_1 \supset N_1 \subset \cdots \subset M_m \supset N_m = \rightred(p)]], \end{equation}
where each $[[N_i \subset M_{i+1} \supset N_{i+1}]]$ is a reduced expression for an atomic coset $\at_{i+1}$. In particular, $p = \at_1.\at_2.\cdots . \at_m$. Any core coset has an atomic reduced expression, see \cite[Cor. 2.17]{KELP3} and thus any coset has an atomic-factored reduced expression by \cite[Proposition 4.28]{EKo}.


We have
\begin{equation} \label{BSIexplicit} \BS(I_{\bullet}) = R^{\leftred(p)} \ot_{R^{M_1}} R^{N_1} \ot_{R^{M_2}} \cdots \ot_{R^{M_m}} R^{\rightred(p)} \end{equation}
viewed as an $(R^I,R^J)$-bimodule. Meanwhile, $\BS(I'_{\bullet})$ is the same abelian group,
but is viewed as an $(R^{\leftred(p)},R^{\rightred(p)})$-bimodule.  There is an action of each $R^{N_i}$ on $\BS(I_{\bullet})$ by multiplication in the $i$-th tensor factor of  \eqref{BSIexplicit}. Indeed, this induces an injective map
\begin{equation}\label{BSIactsonitself} \BS(I_{\bullet}) \to \End(\BS(I_{\bullet}))
\end{equation}
which is not surjective in general.

An arbitrary reduced expression for $p$ might never factor through the subset $\leftred(p)$ or $\rightred(p)$. The first advantage of an atomic-factored expression is that there is an obvious action of $R^{\leftred(p)}$ on $\BS(I_{\bullet})$ by left-multiplication, and an obvious action of $R^{\rightred(p)}$ by right-multiplication. 
The goal is to prove that these two actions agree up to a twist by $\mi{p}$, modulo lower terms.

We denote by $\Id_{I_\bullet}$ the identity morphism of $\BS(I_\bullet)$.
\begin{defn}\label{def:generalatomic} Let  $I_{\bullet}$ be an atomic-factored reduced expression as in \eqref{specialexpression}. We say that \emph{polynomial forcing} holds for $I_{\bullet}$  if for all $f \in R^{\rightred(p)}$, within $\End(\BS(I_{\bullet}))$ as described in \eqref{BSIexplicit} and \eqref{BSIactsonitself}, we have 
\begin{equation} \label{eq:polyforcingfg} \mi{p}(f) \cdot \Id_{I_\bullet}  \equiv  \Id_{I_\bullet}\cdot f \text{ modulo } \End_{< p}(\BS(I_{\bullet})). \end{equation}
We say that \emph{polynomial forcing} holds for a double coset $p$ if it holds for all atomic-factored reduced expressions $I_\bullet$  satisfying $I_\bullet \expr p$.\end{defn}

This definition generalizes Definition \ref{defn:polyforceforp} because atomic cosets have only one reduced expression.

Let $p$ be an arbitrary $(I,J)$-coset.  Since $\leftred(p) \subset I$, there is an inclusion of rings $R^I \subset R^{\leftred(p)}$, and $R^{\leftred(p)}$ is naturally an $R^I$-module. Similarly, $R^{\rightred(p)}$ is an $R^J$-module. If $f \in R^{\rightred(p)}$ then $\mi{p}(f) \in R^{\leftred(p)}$. We can identify the rings $R^{\leftred(p)}$ and $R^{\rightred(p)}$ via $\mi{p}$. In this way, $R^{\leftred(p)}$ becomes an $(R^I, R^J)$-bimodule.

\begin{defn} Let $p$ be an $(I,J)$-coset. The \emph{standard bimodule associated to $p$}, denoted $R_p$, is 
$R^{\leftred(p)}$ as a left $R^I$-module. If $f \in R^J$ and $m \in R_p$ then
\begin{equation} m \cdot f := \mi{p}(f) m. \end{equation}
\end{defn}

We identify $R_p$ with either $R^{\leftred(p)}$ (with right action twisted) or $R^{\rightred(p)}$ (with left action twisted), as is more convenient.

Let $\quot:\End(\BS(I_\bullet))\to \End(\BS(I_{\bullet})) / \End_{< p}(\BS(I_{\bullet}))$ denote the quotient map.

\begin{lem}\label{lem:quotientmap} Let $p$ be a core $(I,J)$-coset and let $I_\bullet$ be a reduced expression for $p$. The bimodule map 
\begin{equation}\label{standardasquotientmorphisms} R_p \to \End(\BS(I_{\bullet})) / \End_{< p}(\BS(I_{\bullet})), \qquad  1 \mapsto \quot(\Id_{I_\bullet}) \end{equation} 
is well-defined if and only if polynomial forcing holds for $I_{\bullet}$. 
\end{lem}
\begin{proof} 
 
The right action of $f \in R^J$ on $1 \in R^J = R_p$ yields $f \in R^J$, and the right action on $\quot(\Id_{I_\bullet})$ yields $\quot(\Id_{I_\bullet} \cdot f)$. The left action of $\mi{p}(f) \in R^I$ on $1 \in R^J = R_p$ yields $f \in R^J$, and the left action on $\quot(\Id_{I_\bullet})$ yields $\quot(\mi{p}(f)\cdot  \Id_{I_\bullet})$. These agree if and only if the bimodule map is well-defined, and if and only if \eqref{eq:polyforcingfg} holds. 
\end{proof}


In conclusion, we have shown the equivalence of three ideas (for almost SW-realizations) for an atomic coset $\at$: the well-definedness of the morphism \eqref{standardasquotientmorphisms} when $p=\at \expr I_\bullet$, atomic polynomial forcing, and the atomic Leibniz rule.


For SW-realizations, \eqref{standardasquotientmorphisms} is an isomorphism by  the theory of singular Soergel bimodules. We can thus prove one of our main theorems.

\begin{thm} \label{thm:main} For an SW-realization, atomic polynomial forcing and atomic Leibniz hold. Moreover, the operators $T_q,T_q'$ in the Leibniz formulas are unique. \end{thm}

\begin{proof}
Assume  $\at\expr [I,M,J]$ is an atomic coset. Then $B_\at\cong \BS([I,M,J])$.

Recall from \cite[\S 4.5]{SingSb} the definition of the submodule $\Gamma_{<\at}B_\at$ of elements supported on lower cosets.
By \cite[Lemma 3.31]{KELP2} we have a short exact sequence 
\begin{equation}\label{sesBa}0\to \End(B_\at,\Gamma_{<\at}B_\at)\to \End(B_\at)\to \Hom(B_\at,B_\at/\Gamma_{<\at}B_\at)\to 0\end{equation}
and, by \cite[Theorem 3.30]{KELP2}, the first term in \eqref{sesBa} is isomorphic to $\End_{<\at}(B_\at)$.
Moreover, since $B_\at$ is indecomposable, by \cite[Theorem 7.10]{SingSb} we have $B_\at/\Gamma_{<\at}B_{\at}\cong R_\at$. The Soergel--Williamson hom formula \cite[Theorem 7.9]{SingSb}
implies that we have an isomorphism $\Hom(B_\at,R_\at)\cong R_\at$\footnote{As in the rest of this paper, we are ignoring degrees here.}
given by $f\mapsto f(1\ot 1)$.
Putting all together, we obtain an isomorphism
\begin{equation}\label{atomicmorphism} \End(B_\at) / \End_{< \at}(B_\at)\stackrel{\sim}{\longrightarrow} R_\at \end{equation}
which sends $\id_\at$ to $1\in R_\at$

By Lemma \ref{lem:quotientmap}, the existence of the isomorphism implies polynomial forcing for $\at$. By Theorem \ref{thm.AL=PF}, this is in turn equivalent to the atomic Leibniz rule for $\at$. Moreover, as proven in \Cref{thm.AL=PF}, the operators $T_q,T_q'$ in the Leibniz formulas are unique. \end{proof}

\begin{rem}
    For a SW-realization, there is an equivalent module-theoretic (rather than morphism-theoretic) version of polynomial forcing. We first recall
 from \cite[Definition 3.23]{KELP3} the filtration on Soergel bimodules
\[ N_{<p}(B)=\sum_{f\in \Hom(\BS(I_\bullet),B),I_\bullet\expr q<p} \im(f).\]
In \cite[Proposition 3.25]{KELP3} we have showed that this coincides with the support filtration $\Gamma_{<p}$ introduced in \cite{SingSb}.

    Let $\at\expr[I,M,J]$ be an atomic coset. We say that \emph{(module-theoretic) polynomial forcing} holds for $\at$ and $f$ if 
    \begin{equation}\label{moduletheoretic}\mi{\at}(f)\ot 1- 1  \ot f\in N_{<\at}(B_\at).
    \end{equation}

    There is an isomorphism $B_\at\cong \End(B_\at)$, where $b\ot b'\in B_\at$ is sent to multiplication by $b\ot b'$. 
Moreover, by \cite[Theorem 3.30]{KELP3} we have
   \[\Hom(B_\at,N_{<\at}B_\at)\cong \End_{<\at}(B_\at).\]
   Hence, \eqref{moduletheoretic} holds if and only if multiplication  by $\mi{\at}(f)\ot 1- 1 \ot  f$ induces a morphism in $\End_{<\at}(B_\at)$, that is, if an only if (morphism-theoretic) polynomial forcing holds for $f$. 
\end{rem}

\subsection{Polynomial forcing: atomic and general} 


In the diagrammatic category, we intend to use the atomic Leibniz rule to prove polynomial forcing, and not vice versa. In that context, polynomial forcing is to be interpreted as the morphism-theoretic statement that \eqref{standardasquotientmorphisms} is a well-defined morphism, when $I_{\bullet}$ is an atomic-factored reduced expression. The goal of this section is to prove that atomic polynomial forcing implies general polynomial forcing. 

In \cite{KELP2}, a compatibility between the Bruhat order and concatenation of reduced expression is proven, which implies the following result.
\begin{prop}[{\cite[Prop. 3.7]{KELP2}}] \label{prop:lowertermslocal}Let $P_{\bullet} \expr p$ and $Q_{\bullet} \expr q$ and $R_{\bullet} \expr r$ be reduced expressions such that $P_{\bullet} \circ Q_{\bullet} \circ R_{\bullet} \expr p.q.r$ is reduced. Then
\begin{equation} \id_{P_{\bullet}} \ot \End_{< q}(Q_{\bullet}) \ot \id_{R_{\bullet}} \subset \End_{< p.q.r}(P_{\bullet} \circ Q_{\bullet} \circ R_{\bullet}). \end{equation} \end{prop}

\begin{thm} \label{thm:polyforcinggeneral} Assume an SW-realization. Then  polynomial forcing holds for all double cosets. \end{thm}

\begin{proof} We first treat the case where $p = p^{\core}$ is a core $(I,J)$-coset. Consider an atomic reduced expression $I_{\bullet}$ for $p$, yielding atomic cosets $\at_i$ such that $p = \at_1 . \at_2.\cdots . \at_m$. Since $\at_i$ are core cosets, by \cite[Lem. 2.10]{KELP3} we have $\mi{p} = \mi{\at_1} \cdot \mi{\at_2} \cdots \mi{\at_m}$. Now within $\BS(I_{\bullet})$ we have 
\begin{equation} \Id_{\at_1} \ot \cdots \ot \Id_{\at_{m}}  f \equiv \Id_{\at_1} \ot \cdots \ot \mi{\at_m}(f) \Id_{\at_m} \equiv \ldots \equiv \mi{\at_1}(\cdots(\mi{\at_m}(f))) \Id_{\at_1}\ot \cdots \ot \Id_{\at_m}, \end{equation}
where $\equiv$ indicates equality modulo lower terms.
At each step we applied polynomial forcing for an atomic coset as proved in \Cref{thm:main}, and used Proposition \ref{prop:lowertermslocal} to argue that lower terms for $\at_i$ embed into lower terms for $p$. Thus polynomial forcing holds for $p$.

If $p$ is not a core coset, let $I_{\bullet}$ be a special reduced expression for $p$ as in \eqref{specialexpression}. Polynomial forcing for $p^{\core}$ implies that 
\[ \mi{p}(f) \cdot\Id_{I_\bullet}\equiv \Id_{I_\bullet} \cdot f \]
modulo 
\[ \id_{[[I \supset \leftred(p)]]} \ot \End_{< p^{\core}}(I'_{\bullet}) \ot \id_{[[\rightred(p) \subset J]]}.\]
By Proposition \ref{prop:lowertermslocal}, they are also equivalent modulo $\End_{<p}(I_{\bullet})$, as desired. \end{proof}

The reader familiar with Soergel bimodules might be familiar with the following example, which showcases how atomic polynomial forcing implies the general case. It also relates our new concept of polynomial forcing to the concept previously in the literature.

\begin{ex} Consider the $(\mt, \mt)$-coset $p = \{s\}$ for a simple reflection $s$, and the reduced expression $I_{\bullet} = [\mt,s,\mt]$. Inside $B_s := \BS(I_{\bullet}) = R \ot_{R^s} R$ we have
\begin{equation} 1\ot f - s(f) \ot 1 = \pa_s(f) \cdot \frac{1}{2}(\alpha_s \ot 1 + 1 \ot \alpha_s). \end{equation}
The term on the right hand side is in $\End_{< s}(I_{\bullet})$, a consequence of the so-called \emph{polynomial forcing relation} in the Hecke category, see e.g. \cite[(5.2)]{Soergelcalculus}.

Now consider the $(\mt, \mt)$ coset $p = \{w\}$ for some $w \in W$, and a reduced expression $I_{\bullet} = [\mt, s_1, \mt, s_2, \ldots, \mt, s_d, \mt]$. By applying the polynomial forcing relation for $B_{s_d}$, we see that $1 \ot \cdots \ot 1 \ot f \equiv 1 \ot \cdots \ot s_d(f) \ot 1$ modulo maps which factor through $[\mt, s_1,\mt, \ldots, \mt, s_{d-1}, \mt]$. Continuing, we can apply polynomial forcing for each $B_{s_i}$ to force $f$ across all the tensors, at the cost of maps which factor through subexpressions of $s_1 \cdots s_d$. By the subexpression property of the Bruhat order, such maps consist of lower terms. \end{ex}

\section{Atomic Leibniz rule in type \texorpdfstring{$A$}{A}}\label{s.A}

In this section we explicitly prove condition \eqref{ALc} in Theorem~\ref{thm.AL=PF} in type $A$. We establish this result for $\Bbbk = \Z$ in this section, rather than over a field.  In addition to extending our results over $\Z$, we feel the ability to be explicit in a key example is its own reward.

In order to achieve this, we first prove in \Cref{thm:demazureoncompletedelta} explicitly an atomic Leibniz rule for a specific set of generators when $W_M$ is the entire symmetric group. In \Cref{ssec:connected} we extend our results to the case when $W_M$ is a product of symmetric groups. This handles all atomic cosets in type $A$. 


\subsection{Notation in type \texorpdfstring{$A$}{A}} \label{subsec:atomicnotation}

We fix notation under the assumption that $W_M$ is an irreducible Coxeter group of type $A$. For $s,t \in M$ we write $\hat{s} := M \setminus \{s\}$ and $\hat{s}\hat{t} := M \setminus \{s,t\}$, etcetera.

Fix $a, b \ge 1$ and let $n = a+b$. Let $W_M = \symm_n$. Let $t = s_a$ and $s = s_b = w_M t w_M$, so that $W_{J} = \symm_a \times \symm_b$ and $W_{I} = \symm_b \times \symm_a$. Let $\at$ be the $(\symm_b \times \symm_a, \symm_a \times \symm_b)$-coset containing $w_M$. The coset $\at$ is depicted as follows, with its minimal element $\mi{\at}$ being the string diagram visible.
\begin{equation} \at = \ig{1}{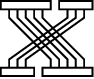}. \end{equation}
Drawn is the example $a=3$ and $b=5$.

For each $0 \le k \le \min(a,b)$, there is an $(\hat{s},\hat{t})$-coset $q_k$ depicted as follows.
\begin{equation} q_1 = \ig{1}{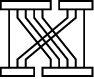}, \qquad q_2 = \ig{1}{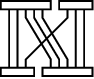}, \qquad q_3 = \ig{1}{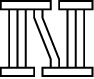}. \end{equation}
Then $p = q_0$, and $\{q_k\}_{0 \le k \le \min(a,b)}$ is an enumeration of all the $(\hat{s},\hat{t})$-cosets. The Bruhat order is a total order in this case:
\[ q_0 > q_1 > q_2 > \ldots > q_{\min(a,b)}. \]

The left redundancy subgroup (see \cite[Section 1.2]{KELP3} to see how to calculate redundancies and cores) of $q_k$ is $\symm_k \times \symm_{b-k} \times \symm_{a-k} \times \symm_k$. For brevity, let $\ell := n - k$. Then (with one exception)  $\leftred(q_k) = \hat{b} \hat{k} \hat{\ell} := M \setminus \{b,k,\ell\}$ and $\rightred(q_k) = \hat{a} \hat{k} \hat{\ell}$. The core of $q_k$ is the double coset depicted as
\begin{equation} q_1^{\core} = \ig{1}{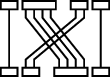}, \qquad q_2^{\core} = \ig{1}{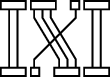}, \qquad q_3^{\core} = \ig{1}{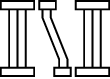}. \end{equation}
Note that the core of $q_k$ is itself atomic, except for $k = \min(a,b)$ when the core is an identity coset.

The case $k = a = b = \ell$ is relatively special. We denote this special coset as $q_{a=b}$.  We have $\leftred(q_{a=b}) = \rightred(q_{a=b}) = I = J$. Unlike other $q_i$, $q_{a=b}$ is core. 


\begin{rem} In type $A$ the following statement is always true: if $\at$ is atomic and $q < \at$ then $q^{\core}$ is either atomic or an identity coset. We do not know for which atomic cosets $\at$ this property holds in other types. \end{rem}

A reduced expression for $q_k$, which factors through the core, is
\begin{equation}\label{qk} q_k \expr [[\hat{b} \supset \hat{b}{\hat{k}\hat{\ell}} \subset \hat{k}\hat{\ell} \supset \hat{a} \hat{k}\hat{\ell} \subset \hat{a}]]. \end{equation}
The exception is when $a = b = k$, in which case
\begin{equation} q_{a=b} \expr [\hat{b}], \end{equation}
that is, the identity expression of $I = \hat{b}$ is a reduced expression for the length zero coset $q_{a=b}$.

Finally, let us write $y_k = \ma{q_k} w_{J}^{-1}$. Then $\pa_{q_k} = \pa_{y_k}$. Here are examples of $y_k$.
\begin{equation}\label{eq.y_i} y_0 = \ig{1}{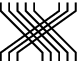}, \qquad
y_1 = \ig{1}{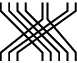}, \qquad
y_2 = \ig{1}{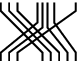}, \qquad
y_3 = \ig{1}{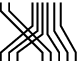}. \end{equation}
Of course $y_0 = \mi{\at}$. Meanwhile $y_k$ is obtained from $y_0$ by removing a $k \times k$ square of crossings from the top. In the special case of the coset $q_{a=b}$, we have $y_{a=b} = e$, the identity of $W$.

\subsection{Complete symmetric polynomials}

Fix $a, b \ge 1$ and continue to use the notation from the previous section. The standard action of $\symm_n$ on $\Z^n$ (with the standard choice of roots and coroots) we call the \emph{permutation realization}. Let $R = \Z[x_1, \ldots, x_n]$. All Demazure operators preserve $R$. By a result of Demazure \cite[Lemme 5]{Demazure}, Frobenius surjectivity holds in type $A$ over $\Z$, that is, for any $I\subset M$  we can find $P_I\in R$ such that $\pa_I(P_I)=1$.

Moreover, for any $J\subset I$ the ring $R^I$ is Frobenius over $R^J$ and we can choose dual bases $\Delta^J_{I,(1)}$ and $\Delta^J_{I,(1)}$ accordingly.

It is well-known that the subring $R^{J} = R^{\symm_a \times \symm_b}$ is generated over $R^{\symm_n}$ by the complete symmetric polynomials \[h_i(x_1, \ldots, x_a)=\sum_{1\leq k_1\leq k_2\leq \ldots \leq k_i\leq a}x_{k_1}x_{k_2}\cdots x_{k_i} \] in the first $a$ variables. In this section, we directly prove the atomic Leibniz rule for $\at$ when $f = h_i(x_1, \ldots, x_a)$.

One of the great features of complete symmetric polynomials is their behavior under Demazure operators. For example, we have 
\begin{equation} \pa_3(h_i(x_1, x_2, x_3)) = h_{i-1}(x_1, x_2, x_3, x_4). \end{equation}
As a consequence, $\pa_2 \pa_3(h_i(x_1, x_2, x_3)) = 0$, a fact which is false if $h_i$ is replaced by some general polynomial inside $R^{\symm_3 \times \symm_{n-3}}$. Indeed, the only elements $w \le s_1 s_2 s_3$ for which $\pa_w(h_i(x_1,x_2,x_3)) \ne 0$ are $w = s_3$ and $w = e$. This will simplify the computation considerably.

Below we shall use letters like $X$, $Y$, and $Z$ to denote subsets of $\{1, \ldots, n\}$. We write $h_i(X)$ for the $i$-th complete symmetric polynomial in the variables $x_j$ for $j \in X$.



\begin{lem}
We have
\begin{equation}\label{eq.delh} \pa_j(h_i(X)) =\begin{cases} h_{i-1}(X \cup \{j+1\})&\text{if }j\in X\text{and }j+1\not\in X\\
	-h_{i-1}(X \cup \{j\})&\text{if }j+1\in X\text{and }j\not\in X\\
	0&\text{otherwise.}\end{cases}\end{equation}
\end{lem}
\begin{proof}
Clearly, $h_i(X)$ is $s_j$-invariant if both $j$ and $j+1$ are inside or outside $X$, so $\pa_j(h_i(X))=0$.
Assume now that $j\in X$ and $j+1\not \in X$.
We have $h_i(X)=h_i(X\setminus \{j\})+h_{i-1}(X)x_j$. Hence, 
\[ \pa_j(h_i(X))=\pa_j(h_{i-1}(X)x_j).\]
Applying the twisted Leibniz rule, by induction on $|X|$, we obtain
\begin{align*}
\pa_j(h_i(X))&=h_{i-2}(X\cup \{j+1\})x_j+s_j(h_{i-1}(X))\pa_j(x_j)\\
&=h_{i-2}(X\cup \{j+1\})x_j+h_{i-1}(X\cup \{j+1\}\setminus \{j\})\\&=h_{i-1}(X\cup \{j+1\}).
\end{align*}

The case where $j+1 \in X$ and $j \notin X$ follows because $\pa_j(s_j(f)) = -\pa_j(f)$, so we have 
\[ \pa_j(h_i(X)) =-\pa_j (h_{i}(X\cup \{j\}\setminus \{j+1\}))=-h_{i-1}(X\cup \{j\}).\qedhere\]
\end{proof}


\begin{thm} \label{thm:demazureoncompletedelta} Use the notation of \S\ref{subsec:atomicnotation}. Fix $a, b \ge 1$, and let  $n = a+b.$ Let $X = \{1, \ldots, a\}$ and $Y = \{n,n-1,\ldots,n+1-a\}$. Recall $y_0, y_1 \in \symm_n$ from \eqref{eq.y_i}.

Then for any $i \ge 0$ and any $g \in R^{\symm_a \times \symm_b}=R^J$ we have
\begin{equation} \label{demazureoncomplete} \pa_{y_0}(h_i(X) \cdot g) = h_i(Y) \cdot \pa_{y_0}(g) + \pa_{y_1}(h_{i-1}(X \cup n) \cdot g). \end{equation}
\end{thm}


As $\pa_p=\pa_{y_0},$ $\underline{p}= y_0, $ and $y_0(h_i(X)) = h_i(Y)$,
this is compatible with \eqref{eq:atomicleibniz},  where $T_{q_1}(h_i(X))= h_{i-1}(X \cup n)$, and $T_{q_k}(h_i(X))$ is zero for all $k > 1$. Most of the terms in \eqref{eq:atomicleibniz} are zero for complete symmetric polynomials, making the formula much easier than the general case.



\begin{proof}
We will do a proof by example, for the example $a = 3$ and $b=4$. The general proof is effectively the same only the notation is more cumbersome.

In this proof, we write $\pa_{123}$ for $\pa_1 \circ \pa_2 \circ \pa_3$ (and not the Frobenius trace associated to the longest element $w_{123}$). We use parenthesization for emphasis, so that $\pa_{(12)3}$ is the same thing as $\pa_{123}$, but emphasizes that $\pa_{(12)3} = \pa_{12} \circ \pa_3$.

Remember that $g$ is invariant under anything except $s_3$, so $\pa_j(g) = 0$ if $j \ne 3$. This implies, for example, that $\pa_3$ kills $\pa_{23}(g)$, and that $\partial_2$ kills $\pa_{3243}(g)$, etcetera.

We claim that
\begin{equation} \label{123ex} \pa_{123}(h_i(123) g) = h_{i-1}(1234) \pa_{23}(g) + h_i(234) \pa_{123}(g).\end{equation}
One proof is to apply the ordinary twisted Leibniz rule repeatedly using \eqref{eq.delh}. After one step we obtain
\[ \pa_{123}(h_i(123) g) = \pa_{12} \left( h_{i-1}(1234) g + h_i(124) \pa_3(g) \right). \]
The first term on the right-hand side is invariant under $s_2$, so it is killed by $\pa_2$. Thus we have
\[ \pa_{123}(h_i(123) g) = \pa_1(\pa_2(h_i(124) \pa_3(g))) = \pa_1(h_{i-1}(1234) \pa_3(g) + h_i(134) \pa_{23}(g)).\]
Again the first term on the right-hand side is invariant under $s_1$, so it is killed by $\pa_1$. A final application of the twisted Leibniz rule to $\pa_1(h_i(134) \pa_{23}(g))$ gives \eqref{123ex}. Essentially, this proof is by iterating the twisted Leibniz rule and arguing that the first term vanishes in every application but the last, because the first term is appropriately invariant. We call this the \emph{easy invariance argument}.


Note that $\pa_{123}(g)$ is invariant under everything but $s_4$. More generally, $\pa_{12\ldots a}(g)$ is killed by $\pa_j$ for all $j \ne a+1$. This is true when $j=1$ since $\pa_1 \pa_1 = 0$. This is true when $j > a+1$ since $\pa_j \pa_{12\ldots a} = \pa_{12\ldots a} \pa_j$ and $\pa_j(g) = 0$. This is true when $2 \le j \le a$ because $\pa_j \pa_{12\ldots a} = \pa_{12\ldots a} \pa_{j-1}$ and $\pa_{j-1}(g) = 0$.

By the easy invariance argument again, but with indices shifted and $g$ replaced by $\pa_{123}(g)$, we have
\begin{equation} \label{234ex} \pa_{234}(h_i(234) \pa_{123}(g)) = h_{i-1}(2345) \pa_{34123}(g) + h_{i}(345) \pa_{234123}(g). \end{equation}
This is how we treat the second term in the right side of \eqref{123ex}.

Note that all computations above are unchanged by adding new variables to our complete symmetric polynomials which are untouched by any of the simple reflections used by the formula. For example, adding $7$ to every $h_i$ in \eqref{123ex} we get
\begin{equation} \label{1237ex} \pa_{123}(h_i(1237) g) = h_{i-1}(12347) \pa_{23}(g) + h_i(2347) \pa_{123}(g).\end{equation}
In this example, we call $7$ an \emph{irrelevant index}.

Now we examine the first term in  the right side of \eqref{123ex}. Note that $\pa_{23}(g)$ is invariant under all simple reflections except $s_1$ and $s_4$. For the next computation, the index $1$ is irrelevant. The easy invariance argument again implies that
\begin{equation} \pa_{234}(h_{i-1}(1234) \pa_{23}(g)) = h_{i-2}(12345) \pa_{3423}(g) + h_{i-1}(1345) \pa_{23423}(g).
\end{equation}
However, as $23423=32434,$ we have $\pa_{23423}(g) = 0$,  so one has the simpler formula
\begin{equation} \label{oneterm} \pa_{234}(h_{i-1}(1234) \pa_{23}(g)) = h_{i-2}(12345) \pa_{3423}(g). \end{equation}

Overall, we see that
\begin{eqnarray} \label{foofoofoo} \pa_{(234)(123)}(h_i(123)g) &=& h_{i}(345) \pa_{(234)(123)}(g) \\ \nonumber &+& h_{i-1}(2345) \pa_{(34)(123)}(g) \\ \nonumber &+& h_{i-2}(12345) \pa_{(34)(23)}(g). \end{eqnarray}
The pattern is relatively straightforward.
Here's the next one in the pattern:
\begin{eqnarray} \label{barbarbar} \pa_{(345)(234)(123)}(h_i(123)g) &=& h_{i}(456) \pa_{(345)(234)(123)}(g) \\ \nonumber &+& h_{i-1}(3456) \pa_{(45)(234)(123)}(g) \\ \nonumber &+& h_{i-2}(23456) \pa_{(45)(34)(123)}(g) \\ \nonumber &+& h_{i-3}(123456) \pa_{(45)(34)(23)}(g). \end{eqnarray}
The word whose Demazure is applied to $g$ is obtained from the concatenation of triples $(345)(234)(123)$ by removing the first index from some of the triples; more specifically, from a prefix of the set of triples. The indices that get removed are instead added to the complete symmetric polynomial. The reason triples appear is because $a = 3$.

The inductive proof of this pattern is the same as above. One takes \eqref{foofoofoo} and applies $\pa_{345}$. The first term splits in two, giving the first two terms of \eqref{barbarbar}, similar to \eqref{123ex} or \eqref{234ex}. Each other term contributes one term in \eqref{barbarbar}, similar to \eqref{oneterm}.

Note that we could have added the irrelevant index $7$ to every set in sight within \eqref{barbarbar}, without any issues. This will be important later.

Repeating until one applies $\pa_{y_0}$, we calculate $\pa_{y_0}(h_i(X)\cdot g)$:
\begin{eqnarray} \label{applypay0} \pa_{(456)(345)(234)(123)}(h_i(123)g) & = & h_i(567) \pa_{(456)(345)(234)(123)}(g) \\ \nonumber & + & h_{i-1}(4567) \pa_{(56)(345)(234)(123)}(g) \\ \nonumber & + & h_{i-2}(34567)\pa_{(56)(45)(234)(123)}(g) \\ \nonumber & + & h_{i-3}(234567) \pa_{(56)(45)(34)(123)}(g) \\ \nonumber & +  & h_{i-4}(1234567) \pa_{(56)(45)(34)(23)}(g). \end{eqnarray}
The fact that there were four triples is because $b=4$. Note that the first term in the RHS is $h_i(Y)\pa_{y_0}(g).$ 

Let us now compute $\pa_{y_1}(h_{i-1}(1237) \cdot g)$.  Note that $y_1 = (56)(345)(234)(123)$. To compute $\pa_{(345)(234)(123)}(h_{i-1}(1237) g)$, we take \eqref{barbarbar}, add the irrelevant index $7$ to all variable lists, and reduce $i$ by one. Now we need only apply $\pa_{(56)}$ to the result. The key thing to note here is that each $h_{\bullet}(\cdots 567)$ is invariant already under $s_5$ and $s_6$. Thus both operators in $\pa_{(56)}$ simply apply to the $g$ term. From this we can compute $\pa_{y_1}(h_{i-1}(X\cup n)\cdot g)$:
\begin{eqnarray} \pa_{(56)(345)(234)(123)}(h_{i-1}(1237)g) & = &  h_{i-1}(4567) \pa_{(56)(345)(234)(123)}(g) \\ \nonumber & + & h_{i-2}(34567)\pa_{(56)(45)(234)(123)}(g) \\ \nonumber & + & h_{i-3}(234567) \pa_{(56)(45)(34)(123)}(g) \\ \nonumber & +  & h_{i-4}(1234567) \pa_{(56)(45)(34)(23)}(g). \end{eqnarray}
This exactly matches all terms from \eqref{applypay0} except the first term.
Thus the theorem is proven!
\end{proof}

\begin{thm} \label{thm:typeAoverZ} The atomic Leibniz rule and atomic polynomial forcing both hold for atomic cosets $\at \expr [I,M,J]$ when $W_M = \symm_n$ when $R = \Z[x_1, \ldots, x_n]$. \end{thm}

\begin{proof} Theorem \ref{thm:demazureoncompletedelta} proved property \eqref{ALc} from Theorem \ref{thm.AL=PF} in this case. Thus conditions \eqref{AL} and \eqref{PF} also hold in this case. \end{proof}

\subsection{Reduction to the connected case} \label{ssec:connected}

The previous section proves an atomic Leibniz rule under the assumption $W_M = \symm_n$. Now we do the general case.

Let $W = \symm_n$. An arbitrary atomic coset in $W$ contains the longest element of the reducible Coxeter group $W_M = \symm_{n_1} \times \cdots \times \symm_{n_k}$ where $\sum n_i = n$. It is a coset for $(\hat{s},\hat{t})$, where $s$ and $t$ are simple reflections in the same irreducible component $\symm_{n_i}$ of $W_M$. We can prove polynomial forcing for an arbitrary atomic coset in type $A$ if we can bootstrap the result from $\symm_{n_i}$ to $W_M$. 

In this discussion, there is no difference between type $A$ and a general Coxeter type. Thus let $M$ be finitary, with $s, t \in M$ and $w_M s w_M = t$. Let $\at$ be the atomic $(\hat{s},\hat{t})$-coset containing $w_M$.

Now suppose that $M = C_1 \sqcup C_2 \sqcup \ldots \sqcup C_k$ is a disjoint union of connected components (the simple reflections in $C_i$ commute with those in $C_j$ for $i \ne j$). Suppose without loss of generality that $s \in C_1$, and let $D = C_2 \sqcup \ldots \sqcup C_k$. Then $t \in C_1$ as well, and $t = w_{C_1} s w_{C_1}$. Let $\at'$ denote the atomic $(C_1 \setminus s, C_1 \setminus t)$-coset containing $w_{C_1}$. Then $\at$ and $\at'$ are related by the operation $+D$ described in \cite[\S 4.10]{EKo}.

\begin{lem}\label{lem.connectedPF} With the notation as above, polynomial forcing holds for $\at$ if and only if it holds for $\at'$. \end{lem}

\begin{lem}\label{lem.connectedAL} With the notation as above, an atomic Leibniz rule holds for $\at$ if and only if it holds for $\at'$. \end{lem}

\begin{proof} The proof is straightforward and left to the reader, but we wish to point out the available ingredients. Many basic properties of the operator $+D$ are given in \cite[\S 4.10]{EKo}. There is a bijection between cosets $q < \at$ and cosets $q' < \at'$, and also a bijection between their reduced expressions. Note that $\pa_{\ma{q}w_{\hat{t}^{-1}}} = \pa_{\ma{q'}w_{C_1\setminus t}^{-1}}$ as operators $R \to R$. Finally, dual bases for the Frobenius extension $R^M \subset R^{M \setminus t}$ can also be chosen as dual bases for the Frobenius extension $R^{C_1} \subset R^{C_1 \setminus t}$. 
\end{proof}




\begin{thm} \label{thm:typeAoverZpart2}
The atomic Leibniz rule and polynomial forcing hold for any atomic coset in type $A_{n-1}$ when $R = \Z[x_1, \ldots, x_n]$.
\end{thm}

\begin{proof} The restriction of the permutation realization to any $\symm_{n_i} \subset \symm_n$  is a $W$-invariant enlargement of the permutation realization of $\symm_{n_i}$.
Thus the result follows from the previous two lemmas, Theorem \ref{thm:typeAoverZ}, and Lemma \ref{lem:changerealization}. \end{proof}

Applying \Cref{lem:changerealization}, we also deduce the atomic Leibniz rule for a host of other realizations, including when $R = \Bbbk[x_1, \ldots, x_n]$ for any commutative ring $\Bbbk$. 


\printbibliography
\end{document}